\documentclass[12pt]{amsart}
\usepackage{latexsym}
\usepackage{amsfonts}
\usepackage{graphicx}
\usepackage{amsfonts}
\usepackage{amsmath,amscd}
\usepackage{color}
\usepackage[colorlinks=true, pdfstartview=FitV, linkcolor=blue, citecolor=blue, urlcolor=blue]{hyperref}
\usepackage{epsfig}
\usepackage{amsthm,amsfonts}
\usepackage{amssymb,graphicx,color}
\usepackage[all]{xy} 
\usepackage{verbatim}
\usepackage{hyperref}
\usepackage{tikz-cd}

\newcommand{\field}[1]{\mathbb{#1}}
\newcommand{\C}{\field{C}}

\newcommand{\K}{\field{K}}

\newcommand{\R}{\field{R}}

\newcommand{\PP}{\field{P}}

\newcommand{\pa}[2]{\frac{\partial #1}{\partial #2}}
\newcommand{\tl}[1]{\tilde{#1}}

\DeclareMathOperator{\Sing}{Sing}

\DeclareMathOperator{\id}{id}
\DeclareMathOperator{\Ree}{Re}
\DeclareMathOperator{\Imm}{Im}

\newtheorem{definition}{Definition}[section]

\newtheorem{theorem}[definition]{Theorem}
\newtheorem{corollary}[definition]{Corollary}

\makeatletter
\@namedef{subjclassname@2020}{%
  \textup{2020} Mathematics Subject Classification}
\makeatother

\title [On quadratic polynomial mappings from $\mathbb{C}^3$ to $\mathbb{C}^2$]{On quadratic polynomial mappings from $\mathbb{C}^3$ to $\mathbb{C}^2$} \makeatletter

\@addtoreset{equation}{section}
\makeatother

\thanks{The author was partially supported by the grant of Narodowe Centrum Nauki number 2019/33/B/ST1/00755.}
\date {\today}

\author{M. Farnik}

\address[M. Farnik]{Jagiellonian University\\
Faculty of Mathematics and Computer Science\\
{\L}ojasie\-wi\-cza~6, 30-348 Krak\'ow, Poland}
\email{michal.farnik@gmail.com}

\thanks{}



\begin{document}

\begin{abstract} We classify quadratic polynomial mappings from $\mathbb{C}^3$ to $\mathbb{C}^2$ up to affine equivalence and topological equivalence. This is a part of a larger project, we have already classified mappings from $\mathbb{C}^2$ to $\mathbb{C}^n$ and inted to classify mappings from $\mathbb{C}^3$ to $\mathbb{C}^3$.
\end{abstract}

\maketitle

\section{Introduction}
Let $\Omega_{\K^k}(d_1,\ldots,d_n)$ denote the space of polynomial mappings $F=(f_1,\ldots,f_n):\K^k\to\K^n$ where $\deg f_i\leq d_i$ for $1\leq i\leq n$. Let $F,G\in \Omega_{\K^k}(d_1,\ldots,d_n)$. We say that $F$ is \emph{topologically} (respectively \emph{affinely} or \emph{linearly}) \emph{equivalent} to $G$ if there are homeomorphisms (respectively affine or linear automorphisms) $\Phi:\K^k\to\K^k$ and $\Psi:\K^n\to\K^n$ such that $F=\Psi\circ G\circ \Phi$. In the papers \cite{fj1,fj2} it was shown that for $\K$ equal $\C$ or $\R$ and a fixed $n>0$ the space $\Omega_{\K^2}(2,\ldots,2)$ of quadratic $\K^2\to\K^n$ mappings splits into a finite number of equivalence classes with respect to affine equivalence (hence also with respect to topological equivalence). Moreover, the authors provided a full classification of mappings in $\Omega_{\K^2}(2,\ldots,2)$. Here we focus on classifying the mappings in $\Omega_{\C^3}(2,2)$. We believe that obtaining the classification of $\Omega_{\R^3}(2,2)$ using similar methods is also possible but significantly more complicated than in the complex case. In the future we intend to classify mappings in $\Omega_{\C^3}(2,2,2)$, classifying quadratic mappings from $\mathbb{C}^3$ to $\mathbb{C}^2$ is a natural first step towards achieving this goal. 

Unlike in the $\K^2\to\K^n$ case, there are infinitely many equivalence classes of quadratic $\C^3\to\C^2$ mappings with respect to affine equivalence. This is quite obvious since $\Omega_{\C^3}(2,2)$ is an affine space of dimension $20$ and the groups of of affine automorphisms of $\C^3$ and $\C^2$ have dimensions $12$ and $6$, respectively. However, $\Omega_{\C^3}(2,2)$ splits into a finite number of equivalence classes with respect to topological equivalence. In fact there is a Zariski open dense subset $U\subset \Omega_{\C^3}(2,2)$ such that every mapping $F\in U$ has the same, generic, topological type. If a mapping $f$ has a generic topological type then we say that $f$ is a \emph{topologically generic} mapping. The affine equivalence classes of topologically generic mappings are parametrized by a two dimensional variety. There are also three topological types with affine equivalence classes parametrized by one dimensional varieties.

Now we present our main result -- we give the complete classification of quadratic polynomial mappings from $\mathbb{C}^3$ to $\mathbb{C}^2$  with respect to affine equivalence and with respect to topological equivalence. The notation and definitions are set up in Section~\ref{sec_def}. We select only one representative for each family of topologically equivalent mappings having infinitely many affine equivalence classes, they are $F_1$, $F_2$, $F_4$ and $F_8$. Items are enumerated according to the topological equivalence. We use the same number and distinct letters to enumerate affine equivalence classes of topologically equivalent mappings. There are $47$ classes for topological equivalence. There are $4$~infinite families and $60$ classes for affine equivalence.

We have the following cases:

\begin{itemize}

\item[(1)] (generic case) $F_1=(x^2+z^2+2y,y^2+z^2+2x+2z)$. The mapping $F_1$ is topologically equivalent to mappings $(x^2+z^2+y,y^2+z^2+\alpha x+\beta z)$ for $\alpha\beta\neq 0$ and $H_1(\alpha,\beta)\neq 0$ (see equation \eqref{F_mult1a}). The critical set $C(F_1)=V(xy-1,x(z+1)-z,yz-(z+1))$ is a smooth curve of degree $3$ parametrized by
$\varphi:\C\setminus\{0,-1\}\ni t\mapsto\left(t(t+1)^{-1},(t+1)t^{-1},t\right)\in C(F_1)$. The mapping $F_1$ has $6$ cusps at points $\varphi(t)$ for $t^3(t+1)^3+t^3-(t+1)^3=0$ and $4$ nodes at points $\varphi(t)$ for $(t+1)^4(t^2+1)^2+2t^2(t+1)^2(t^2-1)+t^4=0$.

\item[(2)] $F_2=(x^2+z^2+y,y^2+z^2+4x+iz)$. The mapping $F_2$ is topologically equivalent to mappings $(x^2+z^2+y,y^2+z^2+\alpha x+\beta z)$ for $\alpha\beta\neq 0$ and $H_1(\alpha,\beta)=0$ (see equation \eqref{F_mult1a}) and $(\alpha^2,\beta^2)\notin\{(1,4),(-4,-1),(-4^{-1},4^{-1})\}$. The critical set $C(F_2)=V(xy-1,x(2z+i)-4z,4yz-(2z+i))$ is a smooth curve of degree~$3$ parametrized by
$\varphi:\C\setminus\{0,-1\}\ni t\mapsto\left(4t(2t+i)^{-1},(2t+i)(4t)^{-1},t\right)\in C(F_2)$. The mapping $F_2$ has $4$ cusps at points $\varphi(t)$ for $(t+1)^4+t^3+6t^2+t=0$, a double cusp at $\varphi(1)$ and $3$ nodes at points $\varphi(t)$ for $(t^2+1)((t+1)^4+2t^3+6t^2+2t)=0$.

\item[(3)] $F_3=(x^2+z^2+y,y^2+z^2+x+2z)$. The critical set $C(F_3)=V(4xy-1,2x(z+1)-z,2yz-(z+1))$ is a smooth curve of degree $3$ parametrized by
$\varphi:\C\setminus\{0,-1\}\ni t\mapsto\left(t(2t+2)^{-1},(t+1)(2t)^{-1},t\right)\in C(F_3)$.
The mapping $F_3$ has two cusps at points $\varphi(t)$ for $t^2+t-1=0$, two double cusps at points $\varphi(t)$ for $t^2+t+2^{-1}=0$ and two nodes at points $\varphi(t)$ for $4(t^2+t)^2=0$.

\item[(4)] $F_4=(x^2+z^2+y,y^2+z^2+2x)$. The mapping $F_4$ is topologically equivalent to mappings $(x^2+z^2+y,y^2+z^2+\alpha x)$ for $\alpha\notin \{-1,0,1\}$. The critical set $C(F_4)=V(2xy-1,(x-1)z,(2y-1)z)$ is reducible and consists of the hyperbola $C_1=V(2xy-1,z)$ and the line $C_2=V(x-1,2y-1)$. The components $C_1$ and $C_2$ intersect at $(1,2^{-1},0)$. The restriction $F_4|_{C_1}$ is injective, the restriction $F_4|_{C_2}$ is generically $2:1$ and is branched at $(1,2^{-1},0)$. The mapping $F_4$ has $3$ cusps at points $(x,(2x)^{-1},0)$ for $4x^3-1=0$, the curves $F_4(C_1)$ and $F_4(C_2)$ intersect at points $F_4(x,(2x)^{-1},0)$ for $x\in\{1,2^{-1},-2^{-1}\}$.

\item[(5)] $F_5=(x^2+z^2+2y,y^2+z^2+2x)$. The critical set $C(F_5)=V(xy-1,(x-1)z,(y-1)z)$ is reducible and consists of the hyperbola $C_1=V(xy-1,z)$ and the line $C_2=V(x-1,y-1)$. The curves $C_1$ and $C_2$ intersect at $(1,1,0)$. The restriction $F_5|_{C_1}$ is injective, the restriction $F_5|_{C_2}$ is generically $2:1$ and is branched at $(1,1,0)$. The mapping $F_5$ has two cusps at points $(\varepsilon,\varepsilon^2,0)$ for $\varepsilon^3=1$ and $\varepsilon\neq 1$, the curves $F_5(C_1)$ and $F_5(C_2)$ intersect at points $F_5(1,1,0)$ and $F_5(-1,-1,0)$.

\item[(6)] $F_6=(x^2+z^2+2y,y^2+z^2)$. The critical set $C(F_6)=V(xy,xz,(y-1)z)$ is reducible and consists the $3$ lines: $C_1=V(x,y-1)$, $C_2=V(x,z)$, $C_3=V(y,z)$. The line $C_2$ intersects $C_1$ and $C_3$, the lines $C_1$ and $C_3$ do not intersect. The restriction $F_6|_{C_2}$ is injective, the restrictions $F_6|_{C_1}$ and $F_6|_{C_3}$ are generically $2:1$ and are branched at $(0,1,0)$ and $(0,0,0)$, respectively. The curves $F_6(C_1)$, $F_6(C_2)$ and $F_6(C_3)$ intersect pairwise at one point, the three points are distinct.

\item[(7)] $F_7=(x^2+z^2,y^2+z^2)$. The critical set $C(F_7)=V(xy,xz,yz)$ is reducible and consists of the $3$ lines: $C_1=V(x,y)$, $C_2=V(x,z)$, $C_3=V(y,z)$. The $3$ lines intersect at $(0,0,0)$. The restrictions $F_7|_{C_i}$, $i=1,2,3$, are generically $2:1$ and are branched at $(0,0,0)$.

\item[(8)] $F_8=(x^2+z^2+2y,yz+x)$. The mapping $F_8$ is topologically equivalent to mappings $(x^2+z^2+2y,yz+x+\alpha y)$ for $\alpha^4\neq -16$. The critical set $C(F_8)=V(xz-1,xy-z,y-z^2)$ is a smooth curve of degree $3$ parametrized by
$\varphi:\C\setminus\{-1\}\ni t\mapsto\left(t^{-1},t^2,t\right)\in C(F_8)$. The mapping $F_8$ has $4$ cusps at points $\varphi(t)$ for $3t^4-1=0$ and $2$ nodes at points $\varphi(t)$ for $t^4+1=0$.

\item[(9)] $F_9=(x^2+z^2+2y,yz+x+Ay)$ for $A=\sqrt{2}(1+i)$. The critical set $C(F_9)=V(x(z+A)-1,xy-z,y-z(z+A))$ is a smooth curve of degree~$3$ parametrized by
$\varphi:\C\setminus\{-1\}\ni t\mapsto\left((t+A)^{-1},t^2+At,t\right)\in C(F_9)$. The mapping $F_9$ has $2$ cusps at points $\varphi(t)$ for $3t^2+7At+17i=0$, a double cusp at $\varphi(-A/2)$ and a node at $\varphi(t)$ for $t^2+At-i=0$.

\item[(10)] $F_{10}=(x^2+z^2+2y,yz+y)$. The critical set $C(F_{10})=V(xy,x(z+1),y-z(z+1))$ is reducible and consists of the parabola $C_1=V(x,y-z(z+1))$ and the line $C_2=V(y,z+1)$. The curves $C_1$ and $C_2$ intersect at $(0,0,-1)$. The restriction $F_{10}|_{C_1}$ is injective, the restriction $F_{10}|_{C_2}$ is generically $2:1$ and is branched at $(0,0,-1)$. The mapping $F_{10}$ has a cusp at $(0,-2/9,-1/3)$, the curves $F_{10}(C_1)$ and $F_{10}(C_2)$ intersect at points $(0,0)$ and $(1,0)$.

\item[(11)] $F_{11}=(x^2+z^2,yz+x+y)$. The critical set $C(F_{11})=V(x(z+1),xy-z,z(z+1))$ is reducible and consists of the hyperbola $C_1=V(xy+1,z+1)$ and the line $C_2=V(x,z)$. The curves $C_1$ and $C_2$ do not intersect. The restrictions $F_{11}|_{C_1}$ and $F_{11}|_{C_2}$ are injective. The curves $F_{11}(C_1)$ and $F_{11}(C_2)$ intersect at points $(0,i)$ and $(0,-i)$.

\item[(12)] $F_{12}=(x^2+z^2,yz+y)$. The critical set $C(F_{12})=V(x(z+1),xy,z(z+1))$ is reducible and consists of the $3$ lines: $C_1=V(x,z)$, $C_2=V(x,z+1)$, $C_3=V(y,z+1)$. The line $C_1$ does not intersect $C_2$ and $C_3$, the lines $C_2$ and $C_3$ intersect at $(0,0,-1)$. The restriction $F_{12}|_{C_1}$ is injective, the restriction $F_{12}|_{C_2}$ is constant and the restriction $F_{12}|_{C_3}$ is generically $2:1$ and branched at $(0,0,-1)$.

\item[(13)] $F_{13}=(x^2+z^2+2y,yz)$. The critical set $C(F_{13})=V(xy,xz,y-z^2)$ is reducible and consists of the parabola $C_1=V(x,y-z^2)$ and the line $C_2=V(y,z)$. The curves $C_1$ and $C_2$ intersect at $(0,0,0)$. The restriction $F_{13}|_{C_1}$ is injective, the restriction $F_{13}|_{C_2}$ is generically $2:1$ and is branched at $(0,0,0)$. The curves $F_{13}(C_1)$ and $F_{13}(C_2)$ intersect only at $(0,0)$.

\item[(14)] $F_{14}=(x^2+z^2,yz+x)$. The critical set $C(F_{14})=V(x(z+1),xy-z,z(z+1))$ is reducible and consists of the double line $C_1=V(x^2,xy-z)$ and the line $C_2=V(y,z)$. The lines $C_1$ and $C_2$ intersect at $(0,0,0)$. The restriction $F_{14}|_{C_1}$ is constant and the restriction $F_{14}|_{C_2}$ is injective.

\item[(15)] $F_{15}=(x^2+z^2,yz)$. The critical set $C(F_{15})=V(x(z+1),xy-z,z(z+1))$ is reducible and consists of the double line $C_1=V(x,z^2)$ and the line $C_2=V(y,z)$. The lines $C_1$ and $C_2$ intersect at $(0,0,0)$. The restriction $F_{15}|_{C_1}$ is constant and the restriction $F_{15}|_{C_2}$ is generically $2:1$ and is branched at $(0,0,0)$.

\item[(16)] $F_{16}=(x^2+y^2+2z,z^2+2x)$. The critical set $C(F_{16})=V(xz-1,y)$ is a hyperbola. The mapping $F_{16}$ has $3$ cusps at points $(\varepsilon,0,\varepsilon^2)$ for $\varepsilon^3=1$.

\item[(17)] $F_{17}=(x^2+y^2,z^2+2x)$. The critical set $C(F_{17})=V(xz,y)$ is reducible and consists of the two lines: $C_1=V(x,y)$, $C_2=V(y,z)$. The lines $C_1$ and $C_2$ intersect at $(0,0,0)$. The restriction $F_{17}|_{C_1}$ is injective, the restriction $F_{17}|_{C_2}$ is generically $2:1$ and is branched at $(0,0,0)$. The curves $F_{17}(C_1)$ and $F_{17}(C_2)$ intersect only at $(0,0)$.

\item[(18)] $F_{18}=(xy+z,z^2+2x)$. The critical set $C(F_{18})=V(x,yz-1)$ is a hyperbola. The restriction $F_{18}|_{C(F_{18})}$ is injective. The image $F_{18}(C(F_{18}))$ is the parabola $V(q-p^2)$ without $(0,0)$.

\item[(19)] $F_{19}=(xy,z^2+2x)$. The critical set $C(F_{19})=V(x,yz)$ is reducible and consists of the two lines: $C_1=V(x,y)$, $C_2=V(x,z)$. The lines $C_1$ and $C_2$ intersect at $(0,0,0)$. The restriction $F_{19}|_{C_1}$ is generically $2:1$ and is branched at $(0,0,0)$, the restriction $F_{19}|_{C_2}$ is constant.

\item[(20)] $F_{20}=(x^2+y^2+2z,z^2)$. The critical set $C(F_{20})=V(xz,yz)$ is reducible and consists of the line $C=V(x,y)$ and the plane $H=V(z)$. The restriction $F_{20}|_{C}$ is injective and the image $F_{20}(H)$ is a line.

\item[(21)] $F_{21}=(x^2+y^2,z^2)$. The critical set $C(F_{21})=V(xz,yz)$ is reducible and consists of the line $C=V(x,y)$ and the plane $H=V(z)$. The restriction $F_{21}|_{C}$ is generically $2:1$ and is branched at $(0,0,0)$, the image $F_{21}(H)$ is a line.

\item[(22)] $F_{22}=(x^2+2yz,y^2+2xy+2z)$. The critical set $C(F_{22})=V(x-y^2,z-y^3-y^2)$ is a smooth curve of degree $3$ parametrized by
$\varphi:\C\setminus\ni t\mapsto\left(t^2,t,t^3+t^2\right)\in C(F_{22})$.
The mapping $F_{22}$ has two cusps at $\varphi(0)$ and $\varphi(-6^{-1})$ and a node at $\varphi(t)$ for $8t^2+4t-1=0$.

\item[(23)] $F_{23}=(x^2+2yz,y^2+2xy+3y/8+2z)$. The critical set $C(F_{23})=V(x-y^2,z-y^3-y^2-3y/8)$ is a smooth curve of degree $3$
parametrized by
$\varphi:\C\setminus\ni t\mapsto\left(t^2,t,t^3+t^2+3t/8\right)\in C(F_{23})$. The mapping $F_{23}$ has a double cusp at $\varphi(-4^{-1})$.

\item[(24a)] $F_{24}=(x^2+2yz,y^2+2xy+2x)$. The critical set $C(F_{24})=V(x(x+y)-(y+1)z,y(y+1),y(x+y))$ is reducible and consists of the parabola $C_1=V(y,x^2-z)$ and the line $C_2=V(x-1,y+1)$. The curves $C_1$ and $C_2$ do not intersect. The restrictions $F_{24}|_{C_i}$, $i=1,2$, are injective. The curves $F_{24}(C_1)$ and $F_{24}(C_2)$ intersect at $(4^{-1},1)$. The mapping is topologically equivalent to $(xy,(y+1)z)$.

\item[(24b)] $F_{37}=(xy,yz+z)$. The critical set $C(F_{37})=V(x(y+1),y(y+1),yz)$ is reducible and consists of the two lines: $C_1=V(x,y)$, $C_2=V(y+1,z)$. The lines $C_1$ and $C_2$ do not intersect. The restrictions $F_{37}|_{C_i}$, $i=1,2$, are injective. The images $F_{37}(C_1)$ and $F_{37}(C_2)$ intersect at $(0,0)$.

\item[(25)] $F_{25}=(x^2+2yz,y^2+2xy+2y)$. The critical set $C(F_{25})=V(x(x+y+1)-yz,xy+y,y^2)$ is reducible and consists of the line $C_1=V(x,y)$ and the double line $C_2=V(x+1+y(z+1),y^2)$. The lines $C_1$ and $C_2$ do not intersect. The restrictions $F_{25}|_{C_i}$, $i=1,2$, are constant.

\item[(26)] $F_{26}=(x^2+2yz,y^2+2xy+2y)$. The critical set $C(F_{26})=V(x^2-yz,xy,y^2)$ is a triple line. The image $F_{26}(C(F_{26}))$ is a point.

\item[(27)] $F_{27}=(x^2+2yz,z^2+2y)$. The critical set $C(F_{27})=V(x,y-z^2)$ is a parabola. The mapping $F_{27}$ has a cusp at $(0,0,0)$.

\item[(28a)] $F_{28}=(x^2+2yz,z^2+2x)$. The critical set $C(F_{28})=V(y,z)$ is a line. The restriction $F_{28}|_{C(F_{28})}$ is injective. The mapping is topologically equivalent to $(x,yz)$.

\item[(28b)] $F_{31}=(x^2+2z,y^2+2z)$. The critical set $C(F_{31})=V(x,y)$ is a line. The restriction $F_{31}|_{C(F_{31})}$ is injective. The mapping is topologically equivalent to $(x,yz)$.

\item[(28c)] $F_{40}=(xy,y^2+2z)$. The critical set $C(F_{40})=V(x,y)$ is a line. The restriction $F_{40}|_{C(F_{40})}$ is injective. The mapping is topologically equivalent to $(x,yz)$.

\item[(28d)] $F_{46}=(x^2+yz,x)$. The critical set $C(F_{46})=V(y,z)$ is a line. The restriction $F_{46}|_{C(F_{46})}$ is injective. The mapping is topologically equivalent to $(x,yz)$.

\item[(28e)] $F_{49}=(x^2+y^2,z)$. The critical set $C(F_{49})=V(x,y)$ is a line. The restriction $F_{49}|_{C(F_{49})}$ is injective. The mapping is topologically equivalent to $(x,yz)$.

\item[(29)] $F_{29}=(x^2+2yz,z^2+2z)$. The critical set $C(F_{29})=V(x(z+1),z(z+1))$ is reducible and consists of the line $C=V(x,z)$ and the plane $H=V(z+1)$. The image $F_{29}(C)$ is a point and the image $F_{29}(H)$ is a line.

\item[(30)] $F_{30}=(x^2+2yz,z^2)$. The critical set $C(F_{30})=V(xz,z^2)$ is the plane $H=V(z)$ with the embedded double line $C=V(x,z^2)$. The image $F_{30}(C)$ is a point and the image $F_{30}(H)$ is a line.

\item[(31a)] $F_{32}=(x^2+z,y^2+x)$ with empty critical set. The mapping is topologically equivalent to $(x,y)$.

\item[(31b)] $F_{41}=(xy+z,y^2+x)$ with empty critical set. The mapping is topologically equivalent to $(x,y)$.

\item[(31c)] $F_{50}=(x^2+y^2+z,x)$ with empty critical set. The mapping is topologically equivalent to $(x,y)$.

\item[(31d)] $F_{52}=(xy+z,x)$ with empty critical set. The mapping is topologically equivalent to $(x,y)$.

\item[(31e)] $F_{56}=(x^2+z,y)$ with empty critical set. The mapping is topologically equivalent to $(x,y)$.

\item[(31f)] $F_{58}=(x^2+y,x)$ with empty critical set. This is $f_{10}$ from \cite{fj1}. The mapping is topologically equivalent to $(x,y)$.

\item[(31g)] $F_{62}=(x,y)$ with empty critical set. This is $f_{12}$ from \cite{fj1}.

\item[(32a)] $F_{33}=(x^2+z,y^2)$. The critical set $C(F_{33})=V(y)$ is a plane. The image $F_{33}(C(F_{33}))$ is a line. The mapping is topologically equivalent to $(x,y^2)$.

\item[(32b)] $F_{42}=(xy+z,y^2)$. The critical set $C(F_{42})=V(y)$ is a plane. The image $F_{42}(C(F_{42}))$ is a line. The mapping is topologically equivalent to $(x,y^2)$.

\item[(32c)] $F_{51}=(x^2+y^2,x)$. This is $f_7$ from \cite{fj1}. The critical set $C(F_{51})=V(y)$ is a plane. The image $F_{51}(C(F_{51}))$ is a parabola. The mapping is topologically equivalent to $(x,y^2)$.

\item[(32d)] $F_{57}=(x^2,y)$. This is $f_9$ from \cite{fj1}. The critical set $C(F_{57})=V(x)$ is a plane. The image $F_{57}(C(F_{57}))$ is a line.

\item[(33)] $F_{34}=(x^2+2y,y^2+2x)$. This is $f_1$ from \cite{fj1}. The critical set $C(F_{34})=V(xy-1)$ is a product of a line and hyperbola. The image $F_{34}(C(F_{34}))$ is a curve with three cusps.

\item[(34)] $F_{35}=(x^2+2y,y^2)$. This is $f_3$ from \cite{fj1}. The critical set $C(F_{35})=V(xy)$ is reducible and consists of two intersecting planes. The image $F_{35}(C(F_{35}))$ consist of a parabola and a tangent line.

\item[(35)] $F_{36}=(x^2,y^2)$. This is $f_4$ from \cite{fj1}. The critical set $C(F_{36})=V(xy)$ is reducible and consists of two intersecting planes. The image $F_{36}(C(F_{36}))$ consist of two intersecting lines.

\item[(36)] $F_{38}=(xy+z,yz)$. The critical set $C(F_{38})=V(xy-z,y^2)$ is a double line. The image $F_{38}(C(F_{38}))$ is a point.

\item[(37)] $F_{39}=(xy,yz)$. The critical set $C(F_{39})=V(xy,y^2,yz)$ is the plane $H=V(y)$ with the embedded double point $C=V(x,y^2,z)$. The image $F_{39}(H)$ is a point.

\item[(38)] $F_{43}=(xy,y^2+2x)$. This is $f_2$ from \cite{fj1}. The critical set $C(F_{43})=V(x-y^2)$ is a product of a line and parabola. The image $F_{43}(C(F_{43}))$ is a curve with one cusp.

\item[(39)] $F_{44}=(xy,y^2+2y)$. This is $f_5$ from \cite{fj1}. The critical set $C(F_{44})=V(y^2+y)$ is reducible and consists of two planes. The image $F_{44}(C(F_{44}))$ consists of a line and a point.

\item[(40)] $F_{45}=(xy,y^2)$. This is $f_6$ from \cite{fj1}. The critical set $C(F_{45})=V(y^2)$ is a double plane planes. The image $F_{45}(C(F_{45}))$ is a point.

\item[(41)] $F_{47}=(x^2+yz,y)$. The critical set $C(F_{47})=V(x,y)$ is a line. The image $F_{47}(C(F_{47}))$ is a point.

\item[(42)] $F_{48}=(x^2+y^2+z^2,0)$ with critical equal $\C^3$. 

\item[(43)] $F_{53}=(xy,x)$. This is $f_8$ from \cite{fj1}. The critical set $C(F_{53})=V(y)$ is a plane. The image $F_{53}(C(F_{53}))$ is a point.

\item[(44a)] $F_{54}=(x^2+y^2+z,0)$ with critical set equal to $\C^3$. The mapping is topologically equivalent to $(x,0)$.

\item[(44b)] $F_{59}=(x^2,x)$ with critical set equal to $\C^3$. This is $f_{11}$ from \cite{fj1}. The mapping is topologically equivalent to $(x,0)$.

\item[(44c)] $F_{60}=(x^2+y,0)$ with critical set equal to $\C^3$. This is $f_{14}$ from \cite{fj1}. The mapping is topologically equivalent to $(x,0)$.

\item[(44d)] $F_{63}=(x,0)$ with critical set equal to $\C^3$. This is $f_{16}$ from \cite{fj1}.

\item[(45)] $F_{55}=(x^2+y^2,0)$ with critical set equal to $\C^3$. This is $f_{13}$ from \cite{fj1}.

\item[(46)] $F_{61}=(x^2,0)$ with critical set equal to $\C^3$. This is $f_{15}$ from \cite{fj1}.

\item[(47)] $F_{64}=(0,0)$ with critical set equal to $\C^3$. This is $f_{17}$ from \cite{fj1}.

\end{itemize}

We will show in Sections \ref{sec_hom}--\ref{sec_specinf} that the classification above enumerates all affine equivalence classes.

In most cases the topological equivalence or the lack thereof is quite obvious. Observe that for any quadratic $\C^3\to\C^2$ mapping having at least one cusp, double cusp or node the topological type of the mapping is uniquely determined by the number of those singularities. For mappings without those singularities we look at the number and types of irreducible components of the critical set (or the scheme associated with the ideal generated by the minors of the Jacobian matrix), at the behavior of the restriction of the mapping to each component of the critical set, at the image of each component and at the intersections of those images. We will now give brief arguments for the least obvious cases. 

\begin{itemize}
\item Both $F_{13}(C(F_{13}))$ and $F_{17}(C(F_{17}))$ have two components intersecting only at $(0,0)$, however, one of the components of $F_{13}(C(F_{13}))$ has a cusp at $(0:0)$ and both components of $F_{17}(C(F_{17}))$ are smooth.
\item We compose $F_{24}$ with $(p/2-q^2/8,q-1)$ to obtain $(y(z+h(2x,y)),y^2+2xy+2x-1)$. Then with $(x/2,y,z-h(x,y))$ to obtain $(yz,(x+y-1)(y+1))$ and with $(z-y+1,y,x)$ to obtain $(xy,(y+1)z)$.

\item We compose $F_{28}$ with $(2x,y,2z)$ and $(p/4,q/4)$ to obtain $(x^2+yz,z^2+x)$. Then with $(x-z^2,y+2xz+z^3,z)$ to obtain $(x^2+yz,x)$ and with $(q,p-q^2)$ to obtain $(x,yz)$.

\item We compose $F_{31}$ with $(p,p-q)$ to obtain $(x^2+2z,x^2-y^2)$ and with $(x,y,(z-x^2)/2)$ to obtain $(z,x^2-y^2)$. Finally compose with $((y+z)/2,(y-z)/2,x)$ to obtain $(x,yz)$.

\end{itemize}

\section{Notation and definitions}\label{sec_def}

Throughout the paper we will consider mappings $F=(f,g):\C^3\rightarrow\C^2$. We take $f=a_1x^2+a_2xy+a_3xz+a_4y^2+a_5yz+a_6z^2+a_7x+a_8y+a_9z+a_{10}$ and, similarly, $g=b_1x^2+\ldots+b_{10}$. Note that $a_i$ and $b_i$ denote coefficients at the appropriate monomials and often change. After composing $F$ with an affine or linear automorphism the coefficients change but we do not introduce new symbols. By $h_x=\pa{h}{x}$ we denote the partial differential of a function $h$ with respect to a variable $x$. We denote $m_{xy}=f_xg_y-f_yg_x$ and similarly for $m_{xz}$ and $m_{yz}$. We denote the critical set of $F$ by $C(F)=V(m_{xy},m_{xz},m_{yz})$, we call the set of critical of $F$ the discriminant of $F$ and denote it by by $\Delta(F)=F(C(F))$.

We denote by $\tl{F}=(\tl{f},\tl{g})$ the mapping obtained by taking the homogeneous parts of maximal degree of the components of a mapping $F=(f,g)$. If both components of $F$ are of degree two then we will often consider $\tl{F}$ as a birational mapping $(\tl{f}:\tl{g}):\PP^2\rightarrow\PP^1$. We take $\tl{m}_{xy}=\tl{f}_x\tl{g}_y-\tl{f}_y\tl{g}_x$, it is the homogeneous part of maximal degree of $m_{xy}$.

We use $x$, $y$ and $z$ for the coordinates in the source $\C^3$ or $\PP^2$ and $p$, $q$ for the coordinates in the target $\C^2$ or $\PP^1.$ The letters $a_i$, $b_i$, $A$, $B$ will denote parameters or coordinates in parameter space.

If there is no danger of confusion we will omit the arguments of a mapping, i.e., we will write $(f_1(x_1,\ldots,x_k),\ldots,f_n(x_1,\ldots,x_k))$ instead of 
$(x_1,\ldots,x_k)\mapsto$\linebreak  $(f_1(x_1,\ldots,x_k),\ldots,f_n(x_1,\ldots,x_k))$.

Let $F:\C^3\rightarrow\C^2$ be a mapping. We say that $P\in C(F)$ is a \emph{cusp} of $F$ if $F^{-1}(F(P))\cap\C(F)=\{P\}$ and there are neighborhoods $U_1$ of $P_1$ and $U_2$ of $F(P)$, open sets $V_1\subset\C^3$ and $V_2\subset\C^2$ and biholomorphic mappings $h_1:(U_1,P)\rightarrow(V_1, 0)$, $h_2(U_2,F(P))\rightarrow(V_2, 0)$ such that $F|_{U_1}=h_2^{-1}\circ F_1\circ h_1$ for $F_1(x,y,z)=(x,y^3+xy+z^2)$. Note that if $P$ is a cusp of $F$ then $P$ is a smooth point of $C(F)$ and a critical point of the restriction $F|_{C(F)}$. Moreover the irreducible component of $\Delta(F)$ containing $F(P)$ is a curve with a cusp singularity at $F(P)$.

Similarly, we call $P\in C(F)$ a \emph{double cusp} if $F$ restricted to some neighborhood of $P$ is biholomorphically equivalent to $(x,y^4+xy+z^2)$.

We say that $P_1,P_2\in C(F)$ are a \emph{node} of $F$ if $F^{-1}(F(P_1))\cap\C(F)=\{P_1,P_2\}$, the points $P_1$ and $P_2$ belong to the same irreducible component of $C(F)$ and there are neighborhoods $U_1$ of $P_1$, $U_2$ of $P_2$ and $U_3$ of $F(P_1)$, open sets $V_1,V_2\subset\C^3$ and $V_3\subset\C^2$ and biholomorphic mappings $h_1:(U_1,P_1)\rightarrow(V_1, 0)$, $h_2(U_2,P_2)\rightarrow(V_2, 0)$ and $h_3(U_3,F(P_1))\rightarrow(V_3, 0)$ such that $F|_{U_1}=h_3^{-1}\circ F_1\circ h_1$ for $F_1(x,y,z)=(x,y^2+z^2)$ and $F|_{U_2}=h_3^{-1}\circ F_2\circ h_2$ for $F_1(x,y,z)=(y^2+z^2,x)$. Note that if $P_1$ and $P_2$ are a node of $F$ then $P_1,P_2$ are smooth points of $C(F)$ and regular points of the restriction $F|_{C(F)}$. Moreover the irreducible component of $\Delta(F)$ containing $F(P_1)$ is a curve with a node singularity at $F(P_1)$.

\section{Classification of homogeneous mappings}\label{sec_hom}

In this section we will consider $F=(f,g)$ such that $f$ and $g$ are homogeneous. We will enumerate up to linear equivalence all possible cases for $F$. First, we will consider the case when $f$ and $g$ are of degree $2$ and one is not a multiple of the other, i.e., $F$ is not equivalent to $(f,0)$. Later we will enumerate the cases with lower degrees.

We begin by taking care of some obvious special cases. First, consider the case when $F$ is equivalent to $(f,z^2)$. If $z$ divides $f$ then $F$ is obviously equivalent to 
\begin{equation}\label{FF_gen8}
\tl{F}_8=(xy,y^2).
\end{equation}
If $z$ does not divide $f$ then $a_1$, $a_2$ or $a_4$ is nonzero. By composing with $(x,y+ax,z)$ for generic $a$ we may assume that $a_1\neq 0$. By composing with $(a_1^{-1/2}x,y,z)$ we obtain $(x^2+a_2xy+a_3xz+a_4y^2+a_5yz+a_6z^2, z^2)$. Composing with $(x-a_2/2 y-a_3/2 z, y,z)$ we obtain $(x^2+a_4y^2+a_5yz+a_6z^2, z^2)$. If $a_4\neq 0$ then we compose with $(x,a_4^{-1/2}y-\sqrt{a_4}a_5/2 z, z)$ to obtain $(x^2+y^2+a_6z^2, z^2)$ and with $(p-a_6 q,q)$ to obtain
\begin{equation}\label{FF_gen3}
\tl{F}_3=(x^2+y^2,z^2).
\end{equation}
Now assume $a_4=0$. If $a_5\neq 0$ then composing with $(x,a_5^{-1}(2y-a_6z),z)$ we obtain
\begin{equation}\label{FF_gen5}
\tl{F}_5=(x^2+2yz,z^2).
\end{equation}
If $a_5=0$ then we compose with $(x,z,y)$ and $(p-a_6 q,q)$ to obtain
\begin{equation}\label{FF_gen6}
\tl{F}_6=(x^2,y^2).
\end{equation}

Now we adopt a geometric approach. We will consider $F$ as a birational $\PP^2\rightarrow\PP^1$ mapping.
It is easy to see that for generic $F$ the sets $V(m_{xy})$ and $V(m_{xz})$ are quadrics, their intersection consists of $4$ points. The curve $V(m_{yz})$ omits the point $V(f_x,g_x)$ and contains the other three. Thus, we will begin with the case when $C(F)$ has three points, then follow with the case of double point and a singular point and other nongeneric cases.

Recall, that we have
$$F=(a_1x^2+a_2xy+a_3xz+a_4y^2+a_5yz+a_6z^2,b_1x^2+b_2xy+b_3xz+b_4y^2+b_5yz+b_6z^2).$$
We compute
$$m_{xy}=(2a_1x+a_2y+a_3z)(b_2x+2b_4y+b_5z)-(a_2x+2a_4y+a_5z)(2b_1x+b_2y+b_3z),$$
$$m_{xz}=(2a_1x+a_2y+a_3z)(b_3x+b_5y+2b_6z)-(a_3x+a_5y+2a_6z)(2b_1x+b_2y+b_3z),$$
$$m_{yz}=(a_2x+2a_4y+a_5z)(b_3x+b_5y+2b_6z)-(a_3x+a_5y+2a_6z)(b_2x+2b_4y+b_5z).$$

Assume that $C(F)$ has $3$ distinct points. If the points are collinear, then by B\'ezout's Theorem each of the three quadrics defining $C(F)$ contains the line passing through the $3$ points. Consequently, $C(F)$ contains a line, which is a special case that we will consider later. Assume that the points are not collinear. By composing $F$ with a linear automorphism of $\PP^2$ we may assume that the points of $C(F)$ are $(1:0:0)$, $(0:1:0)$ and $(0:0:1)$.

We will show that if $F$ is well defined at two distinct points of $C(F)$ and has the same value for both of them then we can reduce to the already solved case $F=(f,z^2)$ or $F=(f,0)$. Indeed, we may assume that $(a_1:b_1)=F(1:0:0)=(1:0)$ and $(a_4:b_4)=F(0:1:0)=(1:0)$, which means that $a_1,a_4\neq 0$ and $b_1=b_4=0$. Moreover, $0=m_{xy}(1:0:0)=2a_1b_2$, $0=m_{xz}(1:0:0)=2a_1b_3$ and $0=m_{yz}(0:1:0)=2a_4b_5$, so $b_2=b_3=b_5=0$. Thus, from now on we may assume that $F$ is injective on all points of $C(F)$ on which it is well defined.

Now we look in how many points of $C(F)$ the birational mapping $F$ is well defined. First consider the case when $F$ is well defined at $3$ points of $C(F)$. By composing with a linear automorphism of $\PP^1$ we may assume that the corresponding points of $F(C(F))$ are $(1:0)$, $(0:1)$ and $(1:1)$.

From $(1:0)=F(1:0:0)=(a_1:b_1)$ we obtain $a_1\neq 0$ and $b_1=0$. Furthermore, $(0:1)=F(0:1:0)=(a_4:b_4)$ and $(1:1)=F(0:0:1)=(a_6:b_6)$ yields $a_4=0$, $b_4\neq 0$ and $a_6=b_6\neq 0$. Replacing $(x,y,z)$ with $(x/\sqrt{a_1},y/\sqrt{b_4},x/\sqrt{a_6})$ we may assume that $a_1=b_4=a_6=b_6=1$.

We have $0=m_{xy}(1:0:0)=2b_2$ and $0=m_{xy}(0:1:0)=2a_2$. Next, $0=m_{xz}(1:0:0)=2b_3$ and $0=m_{xz}(0:0:1)=2a_3-2b_3$. Finally,
$0=m_{yz}(0:1:0)=-2a_5$ and $0=m_{yz}(0:0:1)=2a_5-2b_5$.

Thus we obtain
\begin{equation}\label{FF_gen1}
\tl{F}_1=(x^2+z^2,y^2+z^2).
\end{equation}

Now consider the case when $F$ is well defined at $2$ points of $C(F)$ and not well defined at the third. As before, $F(1:0:0)=(1:0)$ and $F(0:1:0)=(0:1)$ yield $a_1=b_4=1$ and $a_4=b_1=0$. Furthermore, $F(0:0:1)=(0:0)$ means that $a_6=b_6=0$. Next, we have $m_{xy}(1:0:0)=2b_2=0$, $m_{xy}(0:1:0)=2a_2=0$, $m_{xz}(1:0:0)=2b_3=0$ and $m_{yz}(0:1:0)=2a_5=0$. We arrive at $F=(x^2+a_3xz,y^2+b_5yz)$. We have $m_{xy}(0:0:1)=a_3b_5=0$. So at least one of $f$ and $g$ is a square.

Now consider the case when $F$ is well defined at one point of $C(F)$ and not well defined at two other. As before, $F(1:0:0)=(1:0)$ yields $a_1=1$ and $b_1=0$. Furthermore, $F(0:1:0)=F(0:0:1)=(0:0)$ means that $a_4=a_6=b_4=b_6=0$. Next, we have $m_{xy}(1:0:0)=2b_2=0$ and $m_{xz}(1:0:0)=2b_3=0$.
If $b_5=0$ then $F=(f,0)$. So we can assume $b_5\neq 0$ and by dividing $g$ by $b_5$ obtain $b_5=1$. We have $m_{xz}(0:1:0)=a_2=0$ and $m_{xy}(0:0:1)=a_3=0$. We arrive at $F=(x^2+a_5yz,yz)$, composing with $(q,p-a_5q)$ we obtain $F=(yz,x^2)$, so $g$ is a square.

Now, if $F$ is not well defined at $3$ points of $C(F)$ then we obtain $F=(a_2xy+a_3xz+a_5yz,b_2xy+b_3xz+b_5yz)$. The conditions $m_{xy}(0:0:1)=m_{xz}(0:1:0)=m_{yz}(1:0:0)=0$ mean that the vectors $[a_2,a_3,a_5]$ and $[b_2,b_3,b_5]$ are proportional, hence $F$ is equivalent to $(f,0)$.

Now, we assume that $C(F)$ has two points, one of which is double. Note that if the tangent to the double point passes through the smooth point, then by B\'ezout's Theorem $C(F)$ must contain the line through these points. We assume that this is not the case. So we may assume that $(1:0:0)$ and $(0:1:0)$ are critical points and that the line $V(y)$ (but not the whole plane) is tangent to $C(F)$ at $(1:0:0)$. From $m_{xy}(1:0:0)=m_{xz}(1:0:0)=m_{yz}(1:0:0)=0$ we obtain that the vectors $[a_1,a_2,a_3]$ and $[b_1,b_2,b_3]$ are proportional. Thus we may assume that $b_1=b_2=b_3=0$. Since $V(y)$ is tangent to $C(F)$ at $(1:0:0)$ we have $\partial m_{xy}/\partial z (1:0:0)=2a_1b_5=0$ and $\partial m_{xz}/\partial z (1:0:0)=4a_1b_6=0$. If we had $a_1\neq 0$, then $b_5=b_6=0$, so $g=b_4y^2$ would be a square. So $a_1=0$. W have $m_{xy}(0:1:0)=2a_2b_4=0$ and $m_{xz}(0:1:0)=a_2b_5=0$. Again, $a_2\neq 0$ would imply $b_4=b_5=0$ and $g=b_6z^2$, so $a_2=0$. We have $\partial m_{xy}/\partial y (1:0:0)=\partial m_{xz}/\partial z (1:0:0)=0$, since $V(y)$ is tangent but the whole plane not, we must have $\partial m_{yz}/\partial z (1:0:0)=2a_3b_4\neq 0$. By multiplying $f$ and $g$ we may assume $a_3=b_4=1$. Next, $\partial m_{yz}/\partial z (1:0:0)=-b_5=0$. As remarked before, $b_5=0$ implies $b_6\neq 0$. We have $F=(xz+a_4y^2+a_5yz+a_6z^2,y^2+b_6z^2)$, composing with $(p-a_4q,q)$ and $(x-a_5y-(a_6b_4^{-1}-a_4)z,y,b_6^{-1/2}z)$ we obtain $F=(xz,y^2+z^2)$. Permuting the variables we write it down as
\begin{equation}\label{FF_gen2}
\tl{F}_2=(x^2+z^2,yz).
\end{equation}

Now assume that $C(F)$ is a triple point. We may assume that it is $(1:0:0)$ and that $V(y)$ is tangent to $C(F)$ at $(1:0:0)$. We may pick two points at which $F$ is well defined and by composing with linear transformations obtain $F(0:1:0)=(0:1)$ and $F(1:1:0)=(1:0)$. We obtain $a_4=0$, $b_4\neq 0$ and $a_1+a_2\neq 0$. As with the double point, from $(1:0:0)\in C(F)$ we obtain that the vectors $[a_1,a_2,a_3]$ and $[b_1,b_2,b_3]$ are proportional. Thus we may assume that $b_1=b_2=b_3=0$. From the tangency of $V(y)$ at $(1:0:0)$ w obtain $a_1=0$, consequently $a_2\neq 0$. Multiplying $f$ and $g$ by constants we obtain $a_2=1$ and $b_4=1$. Thus, we have $F=(xy+a_3xz+a_5yz+a_6 z^2, y^2+b_5yz+b_6 z^2)$. Composing with $(x-a_5z,y-b_5/2z,z)$ we obtain $F=(xy+a_3xz+a_6 z^2, y^2+b_6 z^2)$. Since $V(y)$ is not a component of $C(F)$ we have $b_6\neq 0$, dividing $z$ by $i\sqrt{b_6}$ we obtain $F=(xy+a_3xz+a_6 z^2, y^2 - z^2)$. Now $C(F)=V(y^2+a_3yz,yz+a_3z^2,xz+a_3xy+2a_6yz)$, so if $a_3^2\neq 1$ then $C(F)$ contains $(1:0:0)$ and $(2a_3a_6/(1-a_3^2):-a_3:1)$ which is a contradiction. If $a_3=-1$ then we multiply $z$ by $-1$ and obtain $F=(xy+xz+a_6 z^2, y^2 - z^2)$. If $a_6=0$ then $y+z$ divides $f$ and $g$, consequently $V(y+z)\in C(F)$ which is a contradiction. So $a_6\neq 0$ and composing with $(a_6x,y,z)$ and $(a_6^{-1}p,q)$ we obtain $F=(xy+xz+z^2, y^2 - z^2)$. This is already a good form, however we prefer to compose it with $(-2z,-y-z,x)$ and obtain
\begin{equation}\label{FF_gen4}
\tl{F}_4=(x^2+2yz,y^2+2xy).
\end{equation}

Now we pass to the case when $C(F)$ does contain the line $V(y)$, but is not the whole plane. If $C(F)$ contains a point outside $V(y)$, then it contains three noncollinear points. According to the analysis above it means that $F$ is equivalent to $(f,z^2)$, in particular to $\tl{F}_3$ if $C(F)$ is a line and a point or $\tl{F}_6$ if $C(F)$ is two lines. It remains to examine the case when $C(F)$ is precisely the line $V(y)$ (as a set). First assume that $F$ is well defined at some point of $C(F)$. We may additionally pick a point outside $C(F)$ for which $F$ is well defined. By composing linear transformations we may assume that $F(1:0:0)=(1:0)$ and $F(0:1:0)=(0:1)$. It follows that $a_1,b_4\neq 0$ and $a_4=b_1=0$ and we may assume that $a_1=b_4=1$. We have $C(F)=V(y)$, so $y$ must divide $m_{xy}$, $m_{xz}$ and $m_{yz}$. The resulting equations yield $b_2=b_3=b_5=b_6=0$, so $g=y^2$. The only mapping with the required geometry is $\tl{F}_5$.

Now assume that $F$ is not well defined along $V(y)$, i.e., $f=yl$ and $g=yk$ for linear forms $l$ and $k$. If $y$, $l$ and $k$ are linearly dependent then $F$ is equivalent to $\tl{F}_8$. If $y$, $l$ and $k$ are linearly independent then $F$ is equivalent to
\begin{equation}\label{FF_gen7}
\tl{F}_7=(xy,yz).
\end{equation}

We sum up the results obtained above in the following corollary:

\begin{corollary}
Let $F=(f,g):\C^3\rightarrow\C^2$ be a mapping with components homogeneous of degree $2$. If $C(F)\neq \C^3$ then the linear equivalence class of $F$ is determined by the critical set $\tl{C}(F)$ of the corresponding $\PP^2\rightarrow\PP^1$ mapping. We have the following possibilities:
\begin{enumerate}
\item $\tl{F}_1=(x^2+z^2,y^2+z^2)$, if $\tl{C}(F)$ is three noncollinear points.
\item $\tl{F}_2=(x^2+z^2,yz)$, if $\tl{C}(F)$ is a double and a smooth point.
\item $\tl{F}_3=(x^2+y^2,z^2)$, if $\tl{C}(F)$ is a line and a point.
\item $\tl{F}_4=(x^2+2yz,y^2+2xy)$, if $\tl{C}(F)$ is a triple point.
\item $\tl{F}_5=(x^2+2yz,z^2)$, if $\tl{C}(F)$ is a line with a nonreduced point.
\item $\tl{F}_6=(x^2,y^2)$, if $\tl{C}(F)$ is two lines.
\item $\tl{F}_7=(xy,yz)$, if $\tl{C}(F)$ is a line.
\item $\tl{F}_8=(xy,y^2)$, if $\tl{C}(F)$ is a double line.
\end{enumerate}
\end{corollary}

Next, we enumerate $F=(f,g)$ such that $f$ is of degree $2$ and $g$ is of a lower degree. It is a classic result that the quadratic form $f$ is equivalent to $x^2+y^2+z^2$, $x^2+y^2$ or $x^2$.

First consider the case where $f=x^2+y^2+z^2$ and $g$ is a linear form. Then $V(f,g)$ is either two points or a double point in $\PP^2$. It is easy to see that in the former case $F$ is equivalent to 
\begin{equation}\label{FF_gen9}
\tl{F}_9=(x^2+yz,x)
\end{equation}
and in the latter case $F$ is equivalent to
\begin{equation}\label{FF_gen10}
\tl{F}_{10}=(x^2+yz,y).
\end{equation}
If $g$ is zero then we simply have
\begin{equation}\label{FF_gen11}
\tl{F}_{11}=(x^2+y^2+z^2,0).
\end{equation}

Next, if $f=x^2+y^2$ and $g$ is a linear form then $V(f)$ is two lines and $V(g)$ a line in $\PP^2$. The line $V(g)$ can meet the lines of $V(f)$ in distinct points or in their intersection point or it can be one of the lines in $V(f)$. Correspondingly we obtain:
\begin{equation}\label{FF_gen12xy}
\tl{F}_{12}=(x^2+y^2,z),\quad \tl{F}_{13}=(x^2+y^2,x),\quad \tl{F}_{14}=(xy,x),\quad \tl{F}_{15}=(x^2+y^2,0).
\end{equation}

If $f=x^2$ then a linear form $g$ can either divide $f$ or not. We obtain:
\begin{equation}\label{FF_gen16xy}
\tl{F}_{16}=(x^2,y),\quad \tl{F}_{17}=(x^2,x),\quad \tl{F}_{18}=(x^2,0).
\end{equation}

Finally, we are left with the simple cases when neither $f$ nor $g$ is of degree $2$. We have:
\begin{equation}\label{FF_gen190}
\tl{F}_{19}=(x,y),\quad \tl{F}_{20}=(x,0),\quad \tl{F}_{21}=(0,0).
\end{equation}

We sum up the results obtained above in the following corollary:

\begin{corollary}
Let $F=(f,g):\C^3\rightarrow\C^2$ be a mapping with $f$ homogeneous of degree $2$ and $g$ homogeneous of degree less than $2$. Then $F$ is linearly equivalent to one of the following:
\begin{enumerate}
\item $\tl{F}_9=(x^2+yz,x)$, $\tl{F}_{10}=(x^2+yz,y)$, $\tl{F}_{12}=(x^2+y^2,z)$, if $C(F)$ is a line.
\item $\tl{F}_{13}=(x^2+y^2,x)$, $\tl{F}_{14}=(xy,x)$, $\tl{F}_{16}=(x^2,y)$, if $C(F)$ is a plane.
\item $\tl{F}_{11}=(x^2+y^2+z^2,0)$, $\tl{F}_{15}=(x^2+y^2,0)$, $\tl{F}_{17}=(x^2,x)$, $\tl{F}_{18}=(x^2,0)$, $\tl{F}_{20}=(x,0)$, $\tl{F}_{21}=(0,0)$, if $C(F)$ is $\C^3$.
\item $\tl{F}_{19}=(x,y)$, if $C(F)$ is empty.
\end{enumerate}
\end{corollary}

\section{Generic behavior at infinity}\label{sec_geninf}

In this section we enumerate the linear and topological equivalence classes of mappings $F$ having general behavior at infinity, i.e., with $\tl{F}=\tl{F}_1$. There are infinitely many linear equivalence classes, the families of topologically equivalent mappings have up to $2$ parameters.

We have
\begin{equation}\label{F_gen00}
F=(x^2+z^2+a_7x+a_8y+a_9z+a_{10},y^2+z^2+b_7x+b_8y+b_9z+b_{10}).
\end{equation}

By composing $F$ with the translation $(x-a_7/2,y-b_8/2,z-a_9/2)$ we may assume that
$$F=(x^2+z^2+a_8y+a_{10},y^2+z^2+b_7x+b_9z+b_{10}).$$

Next, we use the translation $(p-a_{10},q-b_{10})$ on $\C^2$  and obtain
\begin{equation}\label{F_gen0}
F=(x^2+z^2+a_8y,y^2+z^2+b_7x+b_9z).
\end{equation}

Now we have two cases: the generic case when $a_8\neq 0$ and the nongeneric case when $a_8=0$. We will investigate the nongeneric case later and assume now that $a_8\neq 0$. Then we can multiply $x$, $y$ and $z$ by $a_8$ and divide $f$ and $g$ by $a_8^2$ obtaining
\begin{equation}\label{F_gen1}
F=(x^2+z^2+y,y^2+z^2+\alpha x+\beta z),\quad \alpha,\beta\in\C.
\end{equation}

Note that the space $\Omega_{\C^3}(2,2)$ has dimension $20$, whereas the spaces of affine automorphisms on $\C^3$ and $\C^2$ have dimensions $12$ and $6$, respectively. Thus $20-12-6=2$ is the smallest number of parameters we can obtain by simplifying a generic member of $\Omega_{\C^3}(2,2)$ with affine automorphisms.

Again, we will first investigate the generic case when $\alpha,\beta\neq 0$ and leave the nongeneric case $\alpha\beta=0$ for later.

Now let us take $A=\alpha^{-2}$ and $B=\beta^2\alpha^{-2}$, and compose $F$ with $\left(\alpha x/2, y/2, \beta z/2 \right)$ and $(4\alpha^{-2} p, 4\alpha^{-2} q)$. We obtain
\begin{equation}\label{F_gen1a}
F=(x^2+Bz^2+2Ay,Ay^2+Bz^2+2x+2Bz),\quad A,B\in\C^*.
\end{equation}

Observe that $C(F)=V(xy-1,x(z+1)-z,yz-(z+1))$ is a smooth curve of degree $3$. It can be easily parametrized by
\begin{equation}\label{phi}
\varphi:\C\setminus\{0,-1\}\ni t\mapsto\left(\frac{t}{t+1},\frac{t+1}{t},t\right)\in C(F).
\end{equation}

The critical set of $F\circ\varphi$ is equal to the set $V((f\circ\varphi)',(g\circ\varphi)')$ which is equal to the set of zeroes of the derivative of $f\circ\varphi$ and set of zeroes of the polynomial
\begin{equation}\label{Hc}
H_c=Bt^3(t+1)^3+t^3-A(t+1)^3.
\end{equation}

The computation of the selfintersections of $\Delta(F)$ requires some tedious calculations, so we use the computational algebra system Magma \cite{mag} to compute that the set $\{t\in\C\setminus\{0,-1\}: \exists t_1\in\C\setminus\{0,-1\}: t\neq t_1,\ F\circ\varphi(t)=F\circ\varphi(t_1)\}$ is contained in the set of zeroes of the polynomial
\begin{equation}\label{Hn}
H_n=(t+1)^4(Bt^2+A)^2+2t^2(t+1)^2(Bt^2-A)+t^4.
\end{equation}

For generic values of $A$ and $B$ the polynomials $H_c$ and $H_n$ have neither multiple roots nor common roots. In that case $\Delta(F)$ has $6$ cusps at $F\circ\varphi(V(H_c))$ and $4$ nodes at $F\circ\varphi(V(H_n))$ and no other singularities. We use Magma to compute relevant resultants and determine that $H_c$ has a multiple root if and only if $H_n$ has a multiple root if and only if $H_c$ and $H_n$ have a common root if and only if $(A,B)$ is a zero of
\begin{equation}\label{F_mult1}
H_0=(A+B)^4+(A-1)^4+(B+1)^4-A^4-B^4-1+124AB(A-B+1).
\end{equation}

Note that the resultant of $H_c$ and $H_c'$ is actually equal to $-729A^2B^3H_0$, however $AB\neq 0$. Similarly, for $H_n$ and $H_n'$ we obtain $16777216A^8B^{10}H_0$ and for $H_c$ and $H_n$ we obtain $A^4B^4H_0^2$.

Now, substituting $(A,B)$ for $\left(\alpha^{-2},\beta^2\alpha^{-2}\right)$ in equation \eqref{F_mult1} and multiplying the equation by $\alpha^8$ we obtain
\begin{equation}\label{F_mult1a}
H_1=(\alpha^2+\beta^2)^4+(\alpha^2-1)^4+(\beta^2+1)^4-\alpha^8-\beta^8-1+124\alpha^2\beta^2(\alpha^2-\beta^2+1).
\end{equation}

Observe that any $F=(x^2+z^2+y,y^2+z^2+\alpha x+\beta z)$ for $\alpha\beta\neq 0$ is affinely equivalent to $(x^2+Bz^2+2Ay,Ay^2+Bz^2+2x+2Bz)$ with some $AB\neq 0$. Moreover $H_1(\alpha,\beta)\neq 0$ if and only if $H_0(A,B)\neq 0$. We can now prove the following theorem:

\begin{theorem}\label{th_genF}
Let $F=(x^2+z^2+y,y^2+z^2+\alpha x+\beta z)$ for $\alpha\beta\neq 0$ and $H_1(\alpha,\beta)\neq 0$. Then $F$ is topologically equivalent to $F_1=(x^2+z^2+y,y^2+z^2+x+z)$.
\end{theorem}
\begin{proof}
\textbf{Step 1.} Scheme of the proof.

By the argument above it is sufficient to show that the family of mappings $F_{A,B}=(x^2+Bz^2+2Ay,Ay^2+Bz^2+2x+2Bz)$, for $ABH_0(A,B)\neq 0$, is topologically trivial. Let $V=\C^2\setminus V(ABH_0(A,B))$. Since $V$ is connected it is enough to show local topological triviality. Let $\mathcal{F}=(F_{A,B},\id_V):\C^3\times V\rightarrow \C^2\times V$.

We will loosely follow the methods developed in Section 3 of \cite{dj}. We pass to the real structure to construct and integrate real vector fields.  To keep the notation brief we will use the complex coordinates when possible. We identify a complex point $(x_1,\ldots,x_k)$ with the real point $(\Ree x_1,\Imm x_1,\ldots,\Ree x_k,\Imm x_k)$ and a complex mapping $(f_1,\ldots,f_k)$ with a real mapping $(\Ree f_1,\Imm f_1,\ldots,\Ree f_k,\Imm f_k)$.

Fix some $(A_0,B_0)\in V$ and let $K$ be a small closed ball centered at $(A_0,B_0)$ and contained in $V$.
Take $(A,B)\in K\setminus (A_0,B_0)$ and let $\partial_0$ be the constant vector field on $K$ with value $(A-A_0,B-B_0)/||(A-A_0,B-B_0)||$. Let $\gamma:[0,\varepsilon]\rightarrow K$ be the unique integral curve of $\partial_0$ with $\gamma(0)=(A_0,B_0)$. Obviously $\gamma(t_1)=(A,B)$ for $t_1=||(A-A_0,B-B_0)||$. Let $\pi$ be the projection from $\C^2\times K$ to $K$. We will construct vector fields $w_0$ on $\C^2\times K$ and $v_0$ on $\C^3\times K$ such that $d\pi(w_0)=\partial_0$ and $d\mathcal{F}(v_0)=w_0$. Note that we treat here $\pi$ as a restriction of projection $\R^8\rightarrow\R^4$ and $\mathcal{F}$ as a $\R^{10}\rightarrow\R^8$ mapping. If we construct the vector fields correctly, then for a point $(p_0,q_0)\in\C^2$ there will be $\gamma_1$ -- the unique integral curve of $w_0$ such that $\gamma_1(0)=(p_0,q_0,A_0,B_0)$. Then, $\gamma_1(t_1)=(p_1,q_1,A,B)$ for some $(p_1,q_1)\in\C^2$. We define a homeomorphism $\Psi_{A,B}:\C^2\rightarrow \C^2$ by setting $\Psi_{A,B}(p_0,q_0)=(p_1,q_1)$. Similarly, we define a homeomorphism $\Phi_{A,B}:\C^3\rightarrow \C^3$ by setting $\Phi_{A,B}(x_0,y_0,z_0)=(x_1,y_1,z_1)$, where $(x_1,y_1,z_1,A,B)=\gamma_2(t_1)$ and $\gamma_2$ is the unique integral curve of $v_0$ with $\gamma_2(0)=(x_0,y_0,z_0,A_0,B_0)$. By construction we have $F_{A,B}\circ\Phi_{A,B}=\Psi_{A,B}\circ F_{A_0,B_0}$.

There are two conditions that we will need $v_0$ and $w_0$ to satisfy. Firstly, the vector fields will be rugose, this implies that $\gamma_1$ and $\gamma_2$ are unique and $\Phi_{A,B}$ and $\Psi_{A,B}$ are homeomorphisms (see the proof of \cite[Lemma 3.1]{dj}). Secondly, we require that $||w_0(P)||\leq C(||P||+1)$ for all $P\in \C^2\times K$. This condition allows to use the Gronwall Lemma to ensure that $\gamma_1$ does not escape to infinity and is well defined for $t_1$ (see \cite[Claim 3.1]{dj}). Similarly, we need $||v_0(P)||\leq C(||P||+1)$ for certain points that we will expose later.

\textbf{Step 2.} Stratification and rugosity.

We define the stratification
$$\mathcal{S}=\{Y_1,Y_2,Y_3\}=\{\C^2\times V\setminus \Delta(\mathcal{F}), \Delta(\mathcal{F})\setminus \Sing(\Delta(\mathcal{F})), \Sing(\Delta(\mathcal{F}))\}.$$

We refer the reader to  \cite[Subsection 2.1]{dj} for definition and properties of Whitney stratifications. Observe that $\mathcal{S}$ is a Whitney stratification. Moreover, $(p,q,A,B)\in \Delta(\mathcal{F})$ if and only if $(p,q)\in\Delta(F_{A,B})$ and $(p,q,A,B)\in \Sing(\Delta(\mathcal{F}))$ if and only if $(p,q)\in\Sing(\Delta(F_{A,B}))$. See \cite[Lemma 3.5]{fj} for the proof of an analogous fact for $\C^2\rightarrow\C^2$ mappings, the proof in our case is very similar.

We also define a Whitney stratification on $\C^3\times V$:
$$\mathcal{S}'=\{X_1,X_2,X_3,X_4,X_5\}=\{\mathcal{F}^{-1}(Y_1),\mathcal{F}^{-1}(Y_2)\cap C(\mathcal{F}), \mathcal{F}^{-1}(Y_3)\cap C(\mathcal{F})$$
$$\mathcal{F}^{-1}(Y_2)\setminus C(\mathcal{F}), \mathcal{F}^{-1}(Y_3)\setminus C(\mathcal{F})\}.$$

Note that $(x,y,z,A,B)\in C(\mathcal{F})$ if and only if $(x,y,z)\in C(F_{A,B})$ and $(x,y,z,A,B)\in X_3$ if and only if $(x,y,z)\in C(F_{A,B})$ and $F_{A,B}(x,y,z)\in\Sing(\Delta(F_{A,B}))$.

Let us now recall the definition of a rugose vector field. Let $X\subset\C^n$ be a variety with stratification $\{X_i\}$, let $\phi:X\rightarrow\R$ be a real function. We say that $\phi$ is a \emph{rugose function} if the following conditions are fulfilled:
\begin{itemize}
\item The restriction $\phi|_{X_i}$ to any stratum $X_i$ is a smooth function.
\item For any stratum $X_i$ and for any $x\in X_i$ there exists a neighborhood $U$ of $x$ in $\C^n$ and a constant $C>0$ such that for any $y\in X\cap U$ and any $x_1\in X_i\cap U$ we have $|\phi(y)-\phi(x_1)|\leq C||y-x_1||$.
\end{itemize}
A \emph{rugose map} is a map whose components are rugose functions. A vector field $v$ on $X$ is called a rugose vector field if $v$ is a rugose map and $v(x)$ is tangent to the stratum containing $x$ for any $x\in X$.

The stratification $\mathcal{S}$ induces a Whitney stratification on $\C^2\times K$ consisting of the restricted strata. We will construct the vector field $w_0$ so that it will be rugose with respect to the stratification induced by $\mathcal{S}$. For an open subset $U\subset \C^2\times V$ by a rugose vector field we will mean a vector field rugose with respect to the stratification $\{Y_i\cap U\}$. For a biholomorphic mapping $\phi: U\mapsto U'$ we obtain the Whitney stratification $\{\phi(Y_i\cap U)\}$ of $U'$ which we use to define rugose vector fields on $U'$. In this setting, if we construct a rugose vector field on $U'$ then by lifting it to $U$ via $\phi$ we obtain a rugose vector field on $U$. Similarly, the vector field $v_0$ will be rugose with respect to the stratification induced by $\mathcal{S}'$.

\textbf{Step 3.} Constructing the vector fields.

We will construct the vector fields $w_0$ and $v_0$ in several parts and glue them together using a smooth Urysohn's Lemma and smooth partition of unity.

Let $Q_1=(0,0,0,A_0,B_0)$ and $Q_2=(0,0,A_0,B_0)$.

We begin with $X_3$ and $Y_3$. Take a point $P_1=(x_1,y_1,z_1,A_0,B_0)\in X_3$, then $(x_1,y_1,z_1)$ is either a cusp or a node of $F_{A_0,B_0}$. First assume that it is a cusp, so $F_{A_0,B_0}|_{C(F_{A_0,B_0})}$ is not an immersion at $(x_1,y_1,z_1)$. We can express the cusp singularity in its normal form, i.e., we can find neighborhoods $U_1$ of $P_1$ and $U_2$ of $\mathcal{F}(P_1)$, open sets $V_1\subset\C^5$, and $V_2\subset\C^4$ and biholomorphic mappings $h_1:(U_1,P_1)\rightarrow(V_1, Q_1)$, $h_2(U_2,\mathcal{F}(P_1))\rightarrow(V_2, Q_2)$ such that $\mathcal{F}|_{U_1}=h_2^{-1}\circ \mathcal{F}_1\circ h_1$ for $\mathcal{F}_1(x,y,z,A,B)=(x,y^3+xy+z^2,A,B)$. Moreover, the mappings $h_1$, $h_2$ are identity on the $A$ and $B$ coordinates. Thus we obtain the following commutative diagram:  

$$
\begin{tikzcd}
(U_1,P_1) \arrow{r}{\mathcal{F}} \arrow[swap]{d}{h_1} & (U_2,\mathcal{F}(P_1)) \arrow{r}{\pi} \arrow{d}{h_2} & (\C^2, (A_0,B_0)) \arrow{d}{\id}\\
(V_1, Q_1) \arrow{r}{\mathcal{F}_1}& (V_2, Q_2) \arrow{r}{\pi}& (\C^2, (A_0,B_0))
\end{tikzcd}
$$

The construction above was executed over the complex numbers, but now we pass to the real structure with all mappings being analytic. We lift $\partial_0$ trivially (by putting zeroes on first $4$ or $6$ coordinates) to $V_2\cap \C^2\times K$ and $V_1\cap \C^3\times K$. Then we lift via $h_1$ to obtain the vector field $v_1$ on $U_1\cap \C^3\times K$ and via $h_2$ to obtain $w_1$ on $U_2\cap \C^2\times K$. Note that $v_1$ and $w_1$ are rugose, since they are obtained by lifting a constant vector field. Moreover, by taking $K$ small enough we can assume that $U_1$ contains the connected component of $X_3\cap \C^3\times K$ containing $P_1$ and $U_2$ contains the connected component of $Y_3\cap \C^2\times K$ containing $\mathcal{F}(P_1)$.

Now, if $P_3=(x_3,y_3,z_3,A_0,B_0)\in X_3$ is not a cusp of $F_{A_0,B_0}$, then it is a node of $F_{A_0,B_0}$. This means that there is a unique $P_4\in X_3$ such that $\mathcal{F}(P_3)=\mathcal{F}(P_4)$. Again, we will use the normal form of the node multisingularity to lift $\partial_0$. We have biholomorphic mappings $h_3:(U_3,P_3)\rightarrow(V_3, Q_1)$, $h_4:(U_4,P_4)\rightarrow(V_4, Q_1)$, $h_5:(U_5,\mathcal{F}(P_3))\rightarrow(V_5, Q_2)$ such that $\mathcal{F}|_{U_3}=h_5^{-1}\circ \mathcal{F}_3\circ h_3$ for $\mathcal{F}_3(x,y,z,A,B)=(x,y^2+z^2,A,B)$ and $\mathcal{F}|_{U_4}=h_5^{-1}\circ \mathcal{F}_4\circ h_4$ for $\mathcal{F}_4(x,y,z,A,B)=(y^2+z^2,x,A,B)$.

Since $F_{A_0,B_0}$ has only $6$ cusps and $4$ nodes we will shrink $K$ only finitely many times. By shrinking the obtained neighborhoods connected components of $X_3\cap \C^3\times K$ we can assume that they are disjoint. So we can glue the various parts together to obtain the vector field $v_1$ on a set $U_6$ open in $\C^3\times K$ and containing $X_3\cap \C^3\times K$. Similarly, we obtain the vector field $w_1$ on a set $U_7$. Since $Y_3\cap \C^2\times K$ is compact we may assume that there is a constant $C$ such that $||w_1(P)||\leq C$ for all $P\in U_7$. We may need to  shrink $U_6$ and $U_7$ further to maintain the condition $\mathcal{F}(U_6)=U_7$.

Now consider $X_2$ and $Y_2$. We lift $\partial_0$ to $Y_2$ using the horizontal lift, i.e., for $P\in Y_2$ we set $w_2(P)$ to be the unique vector in $T_P Y_2$ which lifts $\partial_0$ and is orthogonal to $\ker d_P(\pi|_{Y_2})$. We need to verify that there is a constant $C$ such that for all $P\in Y_2\cap \C^2\times K$ we have $||w_2(P)||\leq C(||x||+1)$, we will do that in Step 4. Since $\mathcal{F}|_{X_2}$ is an isomorphism on $Y_2$ there is only one way that we can lift the vector field $w_2$ on $Y_2$ to $v_2$ on $X_2$.

Now we would like to extend $v_2$ and $w_2$ to rugose vector fields on neighborhoods of $X_2$ and $Y_2$, respectively, and combine them with $v_1$ and $w_1$. For any $P\in X_2$ there are open sets $U_{P,1}$, $U_{P,2}$, $U_{P,3}$, $U_{P,4}$ and biholomorphic mappings $h_{P,1}:U_{P,1}\rightarrow U_{P,2}$ and $h_{P,2}:U_{P,3}\rightarrow U_{P,4}$ such that $\mathcal{F}|_{U_{P,1}}=h_{P,2}^{-1}\circ \mathcal{F}_3\circ h_{P,1}$ and $h_{P,1}$ and $h_{P,2}$ are identities in the $A$, $B$ coordinates. Note that $h_{P,1}(X_2)$ is given in $U_{P,2}\subset \C^3\times\C^2$ by equations $y=z=0$. Now we pass to the real structure and lift $v_2$ to $v_2'$ on $h_{P,1}(X_2)$. Since $v_2(Q)\in T_Q X_2$ and $v_2$ is a lift of $\partial_0$ we must have $v_2'=(\alpha,\beta,0,0,0,0,\partial_0)$, where $\alpha$ and $\beta$ are smooth functions in variables $\Ree x$, $\Imm x$, $\Ree A$, $\Imm A$, $\Ree B$, $\Imm B$. Let $\pi_1:\R^6\times\R^4\rightarrow\R^2\times\R^4$ be the projection on first two and last four coordinates. We extend $v_2'$ to $U_{P,2}$ by composing it with $\pi_1$. Obviously, $v_2'\circ \pi_1$ is a rugose vector field, we lift it to a rugose vector field $v_{P,3}$ on $U_{P,1}$ via $h_{P,1}$. In the same way we extend $w_2$ to a rugose vector field $w_{P,3}$ on $U_{P,3}$. Note that thanks to the construction $v_3$ is a lift of $w_3$ via $\mathcal{F}$. Since $||w_2(Q)||\leq C(||Q||+1)$ for any $Q\in Y_2\cap U_{P,3}$, we can shrink $U_{P,1}$ and $U_{P,3}$ so that we retain $U_{P,1}=\mathcal{F}^{-1}U_{P,3}$ and have $||w_{P,3}\mathcal{F}(Q)||\leq 2C(||\mathcal{F}(Q)||+1)$ for any $Q\in U_{P,3}$.

Take a smooth partition of unity $1=\sum_i\varphi_i$ on a neighborhood $U_8$ of $Y_2\cap \C^2\times K$ such that for all $i$ there is a point $P_i$ such that $\varphi_i$ is zero outside $U_{P_i,3}$. We define $w_3=\sum_i\varphi_i\cdot w_{P_i,3}$ and $v_3=\sum_i(\varphi_i\circ\mathcal{F})\cdot v_{P_i,3}$. Note that for each $P\in X_2\cap \C^3\times K$ the value $\varphi_i(\mathcal{F}(P))$ is nonzero for finitely many $i$, so $v_3$ is well defined on the neighborhood $\bigcap\{U_{P_i,1}:\ \varphi_i(\mathcal{F}(P))\neq 0\}$ of $P$. Thus $v_3$ is well defined on a neighborhood $U_9$ of $X_2\cap \C^3\times K$ such that $\mathcal{F}(U_9)=U_8$.

Now, take a smooth function $\varphi_1:U_7\cup U_8\rightarrow[0,1]$ such that $\varphi_1$ is equal $1$ on some neighborhood $U_{10}$ of $Y_3\cap \C^2\times K$ and is equal $0$ outside the neighborhood $U_7$ on which $w_1$ is defined. We define $w_4=\varphi_1 w_1+(1-\varphi_1)w_3$ and $v_4=(\varphi_1\circ\mathcal{F}) v_1+(1-(\varphi_1\circ\mathcal{F}))v_3$. By construction $w_4$ is rugose on the neighborhood $U_8\cup U_{10}$ of $\Delta(\mathcal{F})\cap \C^2\times K$ and $v_4$ is rugose on the neighborhood $U_{11}=(U_6\cap\mathcal{F}^{-1}(U_{10}))\cup U_9$ of $C(\mathcal{F})\cap \C^3\times K$.

Now consider $Y_3$. Let $w_5$ be the horizontal lift of $\partial_0$ to $Y_3$, note that this is a constant vector field $(0,\ldots,0,\partial_0)$. Take a smooth function $\varphi_2:\C^2\times K\rightarrow[0,1]$ such that $\varphi_2$ is equal to $1$ on some neighborhood $U_{12}$ of $\Delta(\mathcal{F})\cap \C^2\times K$ and $0$ outside $U_8\cup U_{10}$. We define $w_0=\varphi_1 w_4+(1-\varphi_1)w_5$. Note that by construction $w_0$ is rugose and $||w_0(P)||\leq 2C(||P||+1)$ for any $P\in\C^2\times K$.

We define $v_5$ to be the horizontal lift of $w_0$ to $(\C^3\setminus\C(\mathcal{F}))\times K=(X_1\cup X_4\cup X_5)\cap \C^3\times K$ via $\mathcal{F}$. Note that $\mathcal{F}(X_1)=Y_1$, $\mathcal{F}(X_4)=Y_2$, $\mathcal{F}(X_5)=Y_3$ so rugosity of $w_0$ implies rugosity of $v_5$. Take a smooth function $\varphi_3:\C^3\times K\rightarrow[0,1]$ such that $\varphi_3$ is equal to $1$ on some neighborhood of $\C(\mathcal{F})\cap \C^3\times K$ and $0$ outside $U_{11}$. We define $v_0=\varphi_3 v_4+(1-\varphi_3)v_5$. Note that by construction $v_0$ is rugose.

\textbf{Step 4.} A bound for $||w_2(P)||$.

Since $Y_2\cap \C^2\times K$ is not bounded we must check for critical points at infinity and ensure that we construct vector fields with no trajectories going to infinity. Here we will rely heavily on the methods from \cite[Section 3]{dj}. We lifted $\partial_0$ to $Y_2$ using the horizontal lift and need to verify that there is a constant $C$ such that for all $P\in Y_2\cap \C^2\times K$ we have $||w_2(P)||\leq C(||x||+1)$. By \cite[Lemma 3.2]{dj} it is enough to show that $(||x||+1)\nu(d_P(\pi|_{Y_2}))\geq \frac{1}{C}$, where $\nu$ is the Rabier function. We can find a suitable constant if and only if $\pi|_{Y_2}$ does not have any asymptotic critical values, i.e., there is no sequence of $P_n\in Y_2$ such that $P_n\rightarrow\infty$, $\pi(P_n)$ is convergent and $||P_n||\nu(d_{P_n}(\pi|_{Y_2}))\rightarrow 0$. Due to \cite[Corollary 2.3]{jel} we can replace $\nu(d_{P_n}(\pi|_{Y_2}))$ with $g'(d_{P_n}(\pi),T_{P_n} Y_2)$, which is much easier to compute. We refer the reader to  \cite[Section 2]{jel} for the construction of the $g'$ function, we will only compute its value in our setting.

Take $P\in Y_2$, to compute the value of $g'(d_{P}(\pi),T_{P} Y_2)$ we will need the equation of the hyperplane $T_P Y_2$ in $T_P \C^4$, we represent the linear equation as a row vector of coefficients. Note $\mathcal{F}|_{X_2}$ is an isomorphism of $X_2$ and $Y_2$ since by definition $X_2$ consists of points at which $\mathcal{F}|_{C(\mathcal{F})}$ is an immersion and one to one. So there is a unique $P_0=(x,y,z,A,B)\in X_2$ such that $P=\mathcal{F}(P_0)$. We have

$$d_{P_0}\mathcal{F}=\left[\begin{matrix}2x&2A&2Bz&2y&z^2\\2&2Ay&2B(z+1)&y^2&z^2+2z\\0&0&0&1&0\\0&0&0&0&1\end{matrix}\right].$$

Since $T_P Y_2$ is the image of $d_{P_0}\mathcal{F}$ it is generated by the columns of the matrix representation of $d_{P_0}\mathcal{F}$. It is easy to see that the linear equation which is satisfied by the columns of $d_{P_0}\mathcal{F}$ has coefficients $[1,-x,-y,z-x]$ (we need to use the fact $xy=1$ and $x(z+1)=z$). To compute $g'(d_{P}(\pi),T_{P} Y_2)$ we append the row with the coefficients of the equation of $T_{P} Y_2$ to the matrix of $d_{P}(\pi)$. We obtain the matrix:

$$M=\left[\begin{matrix}0&0&1&0\\0&0&0&1\\1&-x&-y&z-x\end{matrix}\right].$$

By \cite[Definition 2.5]{jel}
$$g'(d_{P}(\pi),T_{P} Y_2)=\max_I\left\{\min_{J\subset I,\ j=1,2}\frac{|M_I|}{|M_J(j)|}\right\}.$$
where $M_I$ are the $3\times 3$ minors of $M$ with columns indexed by $I$ and $M_J(j)$ are the $2\times 2$ minors of $M$ with columns indexed by $J$ and without the $j$-th row. Thus, we have:
$$g'(d_{P}(\pi),T_{P} Y_2)=\max\left\{\min\left\{1,\frac{1}{|z-x|},\frac{1}{|y|}\right\},\min\left\{1,\frac{|x|}{|z-x|},\frac{|x|}{|y|}\right\}\right\}.$$

To compute the asymptotic critical values we take a sequence $P_n= \mathcal{F}(x_n,y_n,z_n,A_n,B_n)$. Since $\pi(P_n)$ must be convergent, $(A_n, B_n)$ converges to some $(A',B')$. The curve $\Delta(F_{A',B'})$ has three points at infinity ($(1:0:0)$, $(0:1:0)$ and $(1:1:0)$), since $P_n\rightarrow\infty$ by taking a subsequence we may assume that $F_{A_n,B_n}(x_n,y_n,z_n)$ tends to one of these points. That means that $(x_n,y_n,z_n)$ tends to one of the three points at infinity of the curve $C(F_{A',B'})$, which are ($(1:0:0:0)$, $(0:1:0:0)$ and $(0:0:1:0)$). This means that $g'(d_{P_n}(\pi),T_{P_n} Y_2)$ is equal to approximately $\min\{|x_n|^{-1},|y_n|^{-1},|z_n|^{-1}\}$. At the same time $||P_n||$ is equal to approximately $\max\{|x|^2,|Ay^2|,\sqrt{2}|Bz^2|\}$. It follows that $||P_n||\nu(d_{P_n}(\pi|_{Y_2}))$ does not converge to $0$, so $\pi|_{Y_2}$ does not have any asymptotic critical values and there is a constant $C$ such that $||w_2(P)||\leq C(||x||+1)$ for all $P\in Y_2\cap \C^2\times K$.

\textbf{Step 5.} A bound for $||v_0(P)||$.

To conclude the proof we need to verify that the curve $\gamma_2$ which we use in Step~1 is well defined, i.e., it does not go to infinity. By construction $\mathcal{F}(\gamma_2(t))=\gamma_1(t)$ and by rugosity of $v_0$ the curve $\gamma_2$ is contained in one of the strata. Consider $(x_0,y_0,z_0,A_0,B_0)\in X_3$, let $X_3'$ be the connected component of $X_3$ containing $(x_0,y_0,z_0,A_0,B_0)$. Since $\mathcal{F}|_{X_3'}$ is an injection we have $\gamma_2(t)=\mathcal{F}|_{X_3'}^{-1}(\gamma_1(t))$. Similarly, if $(x_0,y_0,z_0,A_0,B_0)\in X_2$ then $\gamma_2(t)=\mathcal{F}|_{X_2}^{-1}(\gamma_1(t))$. Thus we may assume that $(x_0,y_0,z_0,A_0,B_0)\notin C(\mathcal{F})$. To use Gronwall Lemma we need to verify that $||v_0(P)||\leq C(||P||+1)$ for all $P\in \mathcal{F}^{-1}(\gamma_1([0,t_1]))$. Suppose this is not true. Then we have a sequence of points $P_n\rightarrow\infty$ such that $||P_n||/||v_0(P_n)||\rightarrow 0$. By compactness of $\gamma_1([0,t_1])$ we may pass to a subsequence and assume that $\mathcal{F}(P_n)\rightarrow (p_1,q_1,A_1,B_1)$. Note that the surfaces given by $x^2+B_1z^2+2A_1y=p_1$ and $A_1y^2+B_1z^2+2x+2B_1z=q_1$ intersect at infinity at points given by $x^2+B_1z^2=A_1y^2+B_1z^2=0$. Thus we may assume that $P_n\rightarrow (a_1b_1:b_1:a_1i:0)$ for some $a_1$, $b_1$ such that $a_1^2=A_1$ and $b_1^2=B_1$.

Take a point $P_0=(x,y,z,A,B)\in\mathcal{F}^{-1}(\gamma_1([0,t_1]))$. For $||P_0||$ big enough $v_0(P_0)$ is the orthogonal lift of $w_0(\mathcal{F}(P_0))$. The three nonzero minors of the matrix associated to $d_{P_0}\mathcal{F}$ are $m_1=4AB(z+1-yz)$, $m_2=4B(xz+x-z)$ and $m_3=4A(xy-1)$ and the kernel of $d_{P_0}\mathcal{F}$ is generated by $m=[m_1,-m_2,m_3,0,0]$. Now we pass to the real structure and consider the point $P_0=(\Ree x, \Imm x,\ldots ,\Ree B, \Imm B)\in \R^{10}$. We will add subscripts $\R$ or $\C$ to highlight which structure we consider. The matrix $d_{P_0}\mathcal{F}_{\R}$ can be obtained from $d_{P_0}\mathcal{F}_{\C}$ by replacing each complex entry $a+bi$ with the real block
$$\left[\begin{matrix}a&-b\\b&a\end{matrix}\right].$$
The kernel of $d_{P_0}\mathcal{F}_{\R}$ is generated  by the two columns of the real matrix corresponding $m$ as one column matrix, i.e., $[\Ree m_1, \Imm m_1, -\Ree m_2, -\Imm m_2, \Ree m_3, \Imm m_3, 0,0,0,0]$ and $[-\Imm m_1, \Ree m_1, \Imm m_2, -\Ree m_2, -\Imm m_3, \Ree m_3,0,0,0,0]$. Now consider the matrix $Mm(P_0)$ obtained by appending $\overline{m}$ as the last row of $d_{P_0}\mathcal{F}_{\C}$:
$$Mm=Mm(P_0)=\left[\begin{matrix}2x&2A&2Bz&2y&z^2\\2&2Ay&2B(z+1)&y^2&z^2+2z\\0&0&0&1&0\\0&0&0&0&1\\ \overline{m}_1&-\overline{m}_2&\overline{m}_3&0&0\end{matrix}\right].$$
Observe that a vector $v\in\R^{10}$ is orthogonal to $\ker d_{P_0}\mathcal{F}_{\R}$ if and only if $Mm_{\R}v$ has zeroes on the last two coordinates. For a vector $w\in\R^8$ let $w'\in\R^{10}$ denote the vector obtained from $w$ by appending two zeroes at the end. Then $v=Mm_{\R}^{-1}w'$ is a vector orthogonal to $\ker d_{P_0}\mathcal{F}_{\R}$ satisfying $d_{P_0}\mathcal{F}_{\R}v=w$, so $v$ is the orthogonal lift of $w$. Thus $v_0(P_0)=Mm_{\R}^{-1}w_0(\mathcal{F}(P_0))'$. Since $w_0(\gamma_1([0,t_1]))$ is compact we have $||v_0(P_0)||\leq C\max\{||Mm_{\R}^{-1}e_i||, i=1,\ldots,8\}$ for some constant $C$.

In order to verify that $||P_n||/||v_0(P_n)||$ does not tend to $0$ we may approximate $P_n$ with $(a_1b_1k_n,b_1k_n,a_1ik_n)$ for some $k_n\in\R$ and $k_n\rightarrow\infty$. Obviously $||P_n||=\Theta(k_n)$, i.e., there is a constant $C_1>0$ such that $C_1^{-1}k_n\leq||P_n||\leq C_1 k_n$ for $n$ big enough. We have 
$$Mm(P_n)_{\C}\approx\left[\begin{matrix}2a_1b_1k_n&2a_1^2&2a_1b_1^2ik_n&2b_1k_n&-a_1^2k_n^2\\
2&2a_1^2b_1k_n&2a_1b_1^2ik_n&b_1^2k_n^2&-a_1^2k_n^2\\0&0&0&1&0\\0&0&0&0&1\\
4\overline{a}_1^3\overline{b}_1^3ik_n^2&4\overline{a}_1^2\overline{b}_1^3ik_n^2&4\overline{a}_1^3\overline{b}_1^2k_n^2&0&0\end{matrix}\right].$$
Observe that $\det Mm(P_n)_{\C}\approx 16a_1^3b_1^2\overline{a}_1^2\overline{b}_1^3ik_n^4$, moreover, the norm of every $4\times 4$ minor of $Mm(P_n)_{\C}$ is in $O(k_n^5)$. It follows that the norm of every entry of $Mm(P_n)^{-1}_{\C}$, and consequently of every entry of $Mm(P_n)^{-1}_{\R}$, is in $O(k_n)$. Consequently $||Mm(P_n)_{\R}^{-1}e_i||\leq C_2 k_n$ for some constant $C_2$ and $n$ big enough. Thus, $||v_0(P_n)||\leq C_1CC_2||P_n||$ for $n$ big enough.
\end{proof}

Now we will consider the case when $H_0(A,B)=0$ (see equation \eqref{F_mult1}), i.e., when $F$ has nongeneric singularities. In our computations we still assume that $AB\neq 0$. We use Magma \cite{mag} to compute the Gr\"obner basis of the ideal $(H_c,\partial H_c/\partial t)$ with respect to lexicographical order with $t$ as the first variable (see equation \eqref{Hc}). The only element of the basis without the variable $t$ is $A\cdot H_0(A,B)$, which confirms that $H_c$ does not have multiple roots if and only if $H_0(A,B)\neq 0$. There are two elements with leading term $t$. They are
$$H_{c1}=tAB\cdot(B-4)(B-1/4)(B+1)+\text{poly}(A,B)$$
and
$$H_{c2}=tA\cdot(15A-28B^3+99B^2+52B-15)+\text{poly}(A,B).$$
The leading coefficient with respect to $t$ of at least one of $H_{c1}$ and $H_{c2}$ is nonzero if and only if $H_c$ has one double root and no other multiple roots. This occurs when $(A,B)\notin\{(1,4),(-1/4,1/4),(-4,-1)\}$. Setting $H_{c1}=0$ or $H_{c2}=0$ allows us to compute the double root. We will denote it by $t_0$ but we will not write it down explicitly since the formula is quite long and does not provide any additional insight.

We also compute the Gr\"obner bases of the ideals $(H_n,\partial H_n/\partial t)$ and $(H_c,H_n)$ and conclude that if $(A,B)\notin\{(1,4),(-1/4,1/4),(-4,-1)\}$ then $t_0$ is a double root of $H_n$ and $H_n$ does not have other multiple roots or common roots with $H_c$. We can now prove the following theorem:

\begin{theorem}\label{th_genF1}
Let $F=(x^2+z^2+y,y^2+z^2+\alpha x+\beta z)$ for $\alpha\beta\neq 0$, $H_1(\alpha,\beta)= 0$ and $(\alpha^2,\beta^2)\notin\{(1,4),(-4,-1),(-1/4,1/4)\}$. Then $F$ is topologically equivalent to $F_2=(x^2+z^2+y,y^2+z^2+4x+iz)$.
\end{theorem}
\begin{proof}
We proceed similarly as in Theorem \ref{th_genF}. Using linear equivalence we transform $F$ to $F_{A,B}=(x^2+Bz^2+2Ay,Ay^2+Bz^2+2x+2Bz)$ (see equation \eqref{F_gen1a}). We consider $V=V(H_0(A,B))\setminus V(AB)\setminus\{(1,4),(-1/4,1/4),(-4,-1)\}$. Recall that $(A,B)=(\alpha^{-2},\beta^2\alpha^{-2})$, so $\alpha$, $\beta$ satisfy the assumptions if and only if $(A,B)\in V$. Again, $V$ is connected so it is enough to show local topological triviality of $F_{A,B}$. Since the proof is very similar to the proof of Theorem \ref{th_genF} we will not repeat it, we will only highlight the differences.

Fix some $(A_0,B_0)\in V$ and let $K$ be a small closed ball centered at $(A_0,B_0)$ such that $K\subset\C^2\setminus V(AB)\setminus\{(1,4),(-1/4,1/4),(-4,-1)\}$ and $K\cap V(H_0(A,B))$ is connected. Take $(A,B)\in K \cap V(H_0(A,B)) \setminus (A_0,B_0)$. Now take a smooth real arc $\gamma:[0,1]\rightarrow K \cap V(H_0(A,B))$ such that $\gamma(0)=(A_0,B_0)$ and $\gamma(1)=(A,B)$. We define the vector field $\partial_0=d\gamma([1])$ on $\gamma$, so $\gamma$ is the integral curve of $\partial_0$.

As in Theorem \ref{th_genF} we lift $\partial_0$ to vector fields $w_0$ on $\C^2\times \gamma$ and $v_0$ on $\C^3\times \gamma$. We use integral curves of $w_0$ and $v_0$ to construct the required homeomorphisms.

The mappings $F_{A,B}$ have now $4$ cusps and $3$ nodes, they also have a double cusp -- a point $P$ such that the germ $(F_{A,B}|_{C(F_{A,B})},P)$ is biholomorphic to the germ of $x\mapsto(x^3,0)$ at $0$. The normal form of a double cusp is $(x,y,z)\mapsto(x, y^4+xy+z^2)$. We refer the reader to \cite{rr} for more information on classification of singularities, the double cusp is enumerated as type $5$ on the list in \cite[Proposition 1.3]{rr}.
\end{proof}

The three remaining cases are covered by the next theorem:

\begin{theorem}\label{th_genF2}
Let $F=(x^2+z^2+y,y^2+z^2+\alpha x+\beta z)$ for $(\alpha^2,\beta^2)\in\{(1,4),(-4,-1),(-1/4,1/4)\}$. Then $F$ is affinely equivalent to $F_3=(x^2+z^2+y,y^2+z^2+x+2z)$.
\end{theorem}
\begin{proof}
Using linear equivalence we transform $F$ to $F_{A,B}=(x^2+Bz^2+2Ay,Ay^2+Bz^2+2x+2Bz)$ for $(A,B)\in \{(1,4),(-1/4,1/4),(-4,-1)\}$ (see equation \eqref{F_gen1a}).

Observe that
$$F_{1,4}=(4p+2,4p-4q+3)\circ F_{-1/4,1/4}\circ (-z,-y+1,-x)$$
and
$$F_{-1/4,1/4}=(-q/4+1/4,-p/4-1/4)\circ F_{-4,-1}\circ (y,x,-z-1).$$
This concludes the proof. Note that to find the affine automorphism $(-z,-y+1,-x)$ we computed the two cusps and two double cusps of $F_{-1/4,1/4}$ and $F_{1,4}$ and then took an affine mapping which transformed one quadruple into the other quadruple. Note that by the cusp of $F_{A,B}$ we mean a point $P\in C(F_{A,B})$ such that $F_{A,B}(P)$ is a cusp of $\Delta(F_{A,B})$. The cusps (double cusps) of $F_{A,B}$ are the points $\varphi(t)$ (see equation \eqref{phi}) where $t$ are the simple (double) roots of $H_c$. 
\end{proof}

Theorems \ref{th_genF}, \ref{th_genF1} and \ref{th_genF2} classify mappings $F$ in the form of equation \eqref{F_gen1} for $\alpha\beta\neq 0$. Let us now consider the case $\alpha\neq 0$, $\beta=0$, i.e., $F=(x^2+z^2+y,y^2+z^2+\alpha x)$. We take $A=\alpha^{-2}$, substitute $(x,y,z)$ for $\left(\alpha x/2, y/2, \alpha z/2 \right)$ and multiply both components of $F$ by $4\alpha^{-2}$. We obtain
\begin{equation}\label{F_gen2}
F_A=(x^2+z^2+2Ay,Ay^2+z^2+2x),\quad A\in\C^*.
\end{equation}

Observe that $C(F_A)$ is given by the ideal $(xy-1,(x-1)z,(y-1)z)$, so it is reducible and consists of the hyperbola $C_1=V(xy-1,z)$ and the line $C_2=V(x-1,y-1)$. The components $C_1$ and $C_2$ intersect at $(1,1,0)$. The mapping $F_A|_{C_2}$ is a double cover of the line $V(p-q-A+1)$ branched at $(1,1,0)$. We parametrize $C_1$ by taking $\varphi_1: \C^*\ni t\mapsto(t,1/t,0)\in C_1$ and obtain $C(F_A\circ\varphi_1)=V(t^3-A)$. To compute the intersections of $F_A(C_1)$ and $F_A(C_2)$ we determine $t\in \C^*$ such that $F_A\circ\varphi_1(t)$ lies on the line $V(p-q-A+1)$. We obtain $(t-1)(t^2-A)=0$. We use Magma to verify that $F_A(C_1)$ does not have any selfintersections.

\begin{theorem}\label{th_genF3}
Let $F=(x^2+z^2+y,y^2+z^2+\alpha x)$ for $\alpha^2\notin\{0,1\}$. Then $F$ is topologically equivalent to $F_4=(x^2+z^2+y,y^2+z^2+2x)$.
\end{theorem}

\begin{proof}
We begin by transforming $F$ into $F_A$ from equation \eqref{F_gen2} with $A\in\C\setminus\{0,1\}$. Next, we proceed as in Theorem \ref{th_genF}. We will not repeat the proof, we will only focus on the important differences, which are the singularities that occur and stratifications that are used. Obviously, here we work over $V=\C\setminus\{0,1\}$ and take $\mathcal{F}=(F_A,\id_V)$. Similarly as in Theorem \ref{th_genF}, to define the stratification $\mathcal{S}$ we take $Y_1=\C^2\times V\setminus\Delta(\mathcal{F})$, however, we split $\Delta(\mathcal{F})\setminus \Sing(\Delta(\mathcal{F}))$ into two connected components. We take $Y_{2a}=\{(p,q,A):\ (p,q)\in F_A(C_1)\}$ and $Y_{2b}=\{(p,q,A):\ (p,q)\in F_A(C_2)\}$, where $C_1$ and $C_2$ are the two connected components of $C(F_A)$ described above. We take $Y_3=\Sing(\Delta(\mathcal{F}))$. Note that $Y_3\cap \C^2\times\{A\}$ consists of $F_A(1,1,0)$, three cusps of $F_A(C_1)$ and two intersections of $F_A(C_1)$ and $F_A(C_2)$ other than $F_A(1,1,0)$. As in Theorem \ref{th_genF}, we define the stratification  $\mathcal{S}'$ of $\C^3\times V$ by taking $\mathcal{F}^{-1}(Y)\cap C(\mathcal{F})$ and $\mathcal{F}^{-1}(Y)\setminus C(\mathcal{F})$ for all strata $Y\in \mathcal{S}$.
Now we exhibit the normal forms of singularities of $F_A$, which are essential in defining the vector fields. For $X_{2a}$ we have folds, as for $X_2$ in Theorem \ref{th_genF}, with normal form $(x,y,z)\mapsto (x,y^2+z^2)$. For $X_{2b}$ the mapping $F_A|_{C_2}$ is $2:1$ and we have a multisingularity with normal form $(x_1,y_1,z_1)\mapsto (x_1,y_1^2+z_1^2)$ and $(x_2,y_2,z_2)\mapsto (x_2,y_2^2+z_2^2)$. For $P\in F_A(C_1)\cap F_A(C_2)\setminus \{F_A(1,1,0)\}$ we have three points in $C(F_A)\cap F_A^{-1}(P)$. For the first two the map germs are the same as for $X_{2b}$ and for the third we have $(x_3,y_3,z_3)\mapsto (y_3^2+z_3^2,x_3)$. The cusps have the normal form $(x,y,z)\mapsto (x,y^3+xy+z^2)$. The singularity at $(1,1,0)$ has normal form $(x,y,z)\mapsto (x,xy^2+y^4+z^2)$. This statement is not obvious and we have not found a suitable reference in the literature. The singularity is not finitely $\mathcal{A}$-determined (since $(x,xy^2+y^4+y^{2k+1}+z^2)$ is not equivalent and has the same $2k$-jet) and therefore is absent in classifications with respect to biholomorphic equivalence. Thus, we will provide a short proof.

First, we subtract second component from the first component and use translations to move the singularity to the origin in source and target. We obtain $F=(x^2-Ay^2, Ay^2+Ay+2x+z^2)$. As earlier, let $C_1$ be the component of $C(F)$ such that $F|_{C_1}$ is generically $1:1$ and $C_2$ be the component of $C(F)$ such that $F|_{C_2}$ is generically $2:1$. Note that $F(C_1)=V(p)$. We will introduce coordinates $p_1$ and $q_1$ so that $F(C_1)=V(p_1)$ and $F(C_1)=V(p_1-q_1^2)$. We use Magma to compute the equation of $F(C_2)$, it is 
\begin{equation}\label{eq_C2}
0=4(A-1)^3p+(A-1)^2q^2+c_1(A)pq+c_2(A)p^2+\text{h.o.t.}
\end{equation}
We set $q_1=(A-1)q+c_1(A)/2(A-1)\cdot p+\text{h.o.t.}$ in such way that the higher order terms in the definition of $q_1$ cancel all higher order terms in equation \eqref{eq_C2} divisible by $q$. Thus we obtain $0=4(A-1)^3p+p^2h_1(p,A)+q_1^2$ as the equation of $F(C_2)$. We take $p_1=-(A-1)^3p-p^2h_1(p,A)/4$ and obtain $F(C_1)=V(4p_1-q_1^2)$. In the $p_1$, $q_1$ coordinates we have
$$F=((1-A)^{3}(x^2-Ay^2)+h_2(x^2,y^2,A), (A-1)(2x+2Ay+z^2)+h_3(x,y,z^2,A)),$$
where $h_2$ and $h_3$ are power series of order $2$ with respect to all variables except the last. Take $z_1=\sqrt{A-1}z+h_4(x,y,z,A)$ so that $h_4$ does not have the $z$ term and $2\sqrt{A-1}zh_4+h_4^2$ coincides with $h_3$ in all terms divisible by $z$. The second component of $F$ is equal to $2(A-1)(x+Ay)+h_5(x,y,A) +z_1^2$, where $h_5$ is of order~$2$ with respect to $x$, $y$ variables. Take $x_1=(A-1)(x+Ay)+h_5(x,y,A)/2$, the first component of $F$ is equal to $(1-A)x_1^2+2A(A-1)^2x_1y-A(A-1)^4y^2+h_6(x_1,y,A)$, where $h_6$ is of order $3$ with respect to $x_1$, $y$ variables. Take $y_1=\sqrt{-A}(A-1)^2y-\sqrt{-A}x_1+h_7(x_1,y,A)$ so that $h_7$ is of order $2$ with respect to $x_1$, $y$ and $y_1^2$ coincides with $2A(A-1)^2x_1y-A(A-1)^4y^2+h_6(x_1,y,A)$ in all terms divisible by $y$. We have
$$F=(x_1^2+y_1^2+h_8(x_1,A),2x_1+z_1^2).$$
We have $C(F)=V(y_1, (2x_1+\partial h_8/\partial x_1)z_1)$, in particular $C_2=V(y_1,z_1)$. Thus $4(x_1^2+h_8(x_1,A))-(2x_1)^2=0$, so $h_8=0$. To obtain the desired normal form it is now enough to take $x_2=2x_1+z_1^2$, $y_2=z_1/\sqrt{2}$, $z_2=y_1$, $p_2=q_1$ and $q_2=p_1-q_1^2/4$.
\end{proof}

As a consequence of passing from equation \eqref{F_gen1} to equation \eqref{F_gen2} we obtain the following:

\begin{corollary}\label{cor_F1}
Let $F=(x^2+z^2+y,y^2+z^2+\alpha x)$ for $\alpha^2=1$. Then $F$ is affinely equivalent to $F_5=(x^2+z^2+2y,y^2+z^2+2x)$.
\end{corollary}

Next we will examine the case $\alpha=0$ in equation \eqref{F_gen1}. We have $F=(x^2+z^2+y,y^2+z^2+\beta z)$, $\beta\in\C$. Observe that
\begin{equation}\label{eq_genF4}
\left(-p+\frac{1}{2},q-p+\frac{1}{4}\right)\circ F\circ\left(iz,-y+\frac{1}{2},ix\right) =(x^2+z^2+y,y^2+z^2+i\beta x).
\end{equation}
As a consequence of equation \eqref{eq_genF4}, Theorem \ref{th_genF3} and Corollary \ref{cor_F1} we obtain the following:
\begin{corollary}
If $F=(x^2+z^2+y,y^2+z^2+\beta z)$, then $F$ is affinely equivalent to $F=(x^2+z^2+y,y^2+z^2+i\beta x)$. Consequently $F$ is
\begin{enumerate}
\item topologically equivalent to $F_4=(x^2+z^2+y,y^2+z^2+2x)$ for $\beta^2\notin\{0,-1\}$,
\item affinely equivalent to $F_5=(x^2+z^2+2y,y^2+z^2+2x)$ for $\beta^2=-1$.
\end{enumerate}
The mapping $F_6=(x^2+z^2+y,y^2+z^2)$ is not topologically equivalent to any of the mappings considered earlier since it is the only one with $C(F)$ being a union of three lines.
\end{corollary}

Now we consider the case $a_8=0$ in equation \eqref{F_gen0}. We have 
\begin{equation}\label{F_genF5}
F=(x^2+z^2,y^2+z^2+\alpha x+\beta z) ,\quad \alpha,\beta\in\C
\end{equation}
For $\alpha\neq 0$ we have 
\begin{equation}\label{eq_genF6}
\left(\frac{q}{\alpha^2}+\frac{\beta^2}{4\alpha^2},\frac{p}{\alpha^2}-\frac{\beta^2}{4\alpha^2}\right)\circ F\circ\left(\alpha y,\alpha x,\alpha z-\frac{\beta}{2}\right) =\left(x^2+z^2+y,y^2+z^2-\frac{\beta}{\alpha} x\right).
\end{equation}
For $\alpha=0$ and $\beta\neq 0$ we have 
\begin{equation}\label{eq_genF7}
\left(\frac{p-q}{\beta^2},\frac{p}{\beta^2}\right)\circ F\circ\left(\beta z,i\beta x,-\beta y\right) =\left(x^2+z^2+y,y^2+z^2\right).
\end{equation}
As a consequence of equation \eqref{eq_genF6} and equation \eqref{eq_genF7} we obtain the following:

\begin{corollary}
Let $F=(x^2+z^2,y^2+z^2+\alpha x+\beta z)$. If $\alpha\neq 0$ then $F$ is affinely equivalent to $F=(x^2+z^2+y,y^2+z^2-\beta/\alpha x)$. Consequently $F$ is
\begin{enumerate}
\item topologically equivalent to $F_4=(x^2+z^2+y,y^2+z^2+2x)$ for $(\beta/\alpha)^2\notin\{0,1\}$,
\item affinely equivalent to $F_5=(x^2+z^2+2y,y^2+z^2+2x)$ for $\beta^2=\alpha^2$.
\item affinely equivalent to $F_6=(x^2+z^2+y,y^2+z^2)$ for $\beta=0$.
\end{enumerate}
Moreover, if $\alpha=0$ and $\beta\neq 0$ then $F$ is affinely equivalent to $F=(x^2+z^2+y,y^2+z^2)$.

The mapping $F_7=(x^2+z^2,y^2+z^2)$ is not topologically equivalent to any of the mappings considered earlier since it is the only one with $C(F)$ being a union of three lines meeting in a point.
\end{corollary}

This concludes the enumeration of equivalence classes for mappings $F$ such that $C(F)$ has $3$ points at infinity.

\section{Double critical point at infinity}\label{sec_geninf2}

In this section we enumerate the linear and topological equivalence classes of mappings $F$ such that $\tl{F}=\tl{F}_2$, i.e., having a double a smooth point at infinity. There are infinitely many linear equivalence classes, there is one one-parameter family of topologically equivalent mappings.

We have
\begin{equation}\label{F_gen3}
F=(x^2+z^2+a_7x+a_8y+a_9z+a_{10},yz+b_7x+b_8y+b_9z+b_{10}).
\end{equation}
By composing with translations we can assume that $a_7=a_9=b_9=a_{10}=b_{10}=0$ and obtain
\begin{equation}\label{F_gen4}
F=(x^2+z^2+a_8y,yz+b_7x+b_8y).
\end{equation}
There are several cases depending on whether the constants in equation \eqref{F_gen4} are nonzero. First, we consider the case $a_8b_7\neq 0$. We compose $F$ with $(\sqrt{2^{-1}a_8b_7}x,b_7y,\sqrt{2^{-1}a_8b_7}z)$ and $(2a_8^{-1}b_7^{-1}p,\sqrt{2a_8^{-1}b_7^{-3}}q)$ and obtain
\begin{equation}\label{F_gen5}
F_A=(x^2+z^2+2y,yz+x+Ay),\quad A\in\C.
\end{equation}

Observe that $C(F_A)=V(x(z+A)-1,xy-z,y-z^2-Az)$. Note that $C(F_A)$ is a smooth curve of degree $3$ and $\overline{C}(F_A)$ has two points at infinity: a smooth point $(1:0:0:0)$ and a singular point $(0:1:0:0)$. Moreover, $C(F_A)$ can be easily parametrized by
\begin{equation}\label{phi1}
\varphi:\C\setminus\{-1\}\ni t\mapsto\left(\frac{1}{t+A},t^2+At,t\right)\in C(F_A).
\end{equation}

The critical set of of $F\circ\varphi$ is equal to the set of zeroes of the derivative of $f\circ\varphi$, which is equal to the set of zeroes of the polynomial
\begin{equation}\label{Hc1}
H_{c1}=(3t+A)(t+A)^3-1.
\end{equation}
The selfintersections of $\Delta(F)$ are the points $F_A\circ\varphi(t)$, where $t$ are zeroes of the polynomial
\begin{equation}\label{Hn1}
H_{n1}=t^2(t+A)^2+1.
\end{equation}
It follows that $F_A$ has $4$ cusps and $2$ nodes if and only if $H_{c1}$ and $H_{n1}$ have no double or common roots. This happens when $A^4\neq -16$.

We can now prove the following theorem:

\begin{theorem}\label{th_genF4}
Let $F_A=(x^2+z^2+2y,yz+x+Ay)$ for $A^4\neq -16$. Then $F_A$ is topologically equivalent to $F_8=(x^2+z^2+2y,yz+x)$.
\end{theorem}
\begin{proof}
Again, the proof is very similar to the proof of Theorem \ref{th_genF}, we will not repeat the proof but just highlight the differences. The only  thing that needs reexamination are the computations in Steps $4$ and $5$ of the original proof.

In Step $4$ we need to compute $g'(d_{P}(\pi),T_{P} Y_2)$. We compute $d_{P_0}\mathcal{F}$ for $P_0=(x,y,z,A)\in X_2\cap\mathcal{F}^{-1}(P)$ and obtain 
$$d_{P_0}\mathcal{F}=\left[\begin{matrix}2x&2&2z&0\\1&z+A&y&y\\0&0&0&1\end{matrix}\right].$$
The linear equation which is satisfied by the columns $d_{P_0}\mathcal{F}$ has coefficients $[1,-2x,2z]$, so we obtain the matrix
$$M=\left[\begin{matrix}0&0&1\\1&-2x&2z\end{matrix}\right].$$
By definition 
$$g'(d_{P}(\pi),T_{P} Y_2)=\max\left\{\min\left\{1,\frac{1}{|2z|}\right\},\min\left\{1,\frac{|x|}{|z|}\right\}\right\}.$$
Thus for $P_n=\mathcal{F}(x_n,y_n,z_n,A_n)$ if $\nu(d_{P_n}(\pi|_{Y_2}))\rightarrow 0$ then $|z_n|\rightarrow\infty$. Since $y_n=z_n^2+A_nz_n$ the norm $||P_n||$ is equal to approximately $|y_n z_n|$, so we cannot have $||P_n||\nu(d_{P_n}(\pi|_{Y_2}))\rightarrow 0$. This had to be shown in this step.

In Step $5$ we need to examine the matrix $Mm(P_n)_{\C}$ for $P_n=(x_n,y_n,z_n,A_n)\in\C^3\times\C$ such that $||P_n||\rightarrow\infty$ but $\mathcal{F}(P_n)$ converges to some $(p,q,A)\in Y_1\cap\C^2\times
\C\setminus\{\sqrt[4]{-16}\}$. Note that any fiber of $F_A$ has $3$ points at infinity: $(1:0:\pm i:0)$ and $(0:1:0:0)$. First, consider the case $(x_n:y_n:z_n)\rightarrow(1:0:\pm i)$. Since $y_n(z_n+A_n)+x_n\rightarrow q$ we must have $y_n\rightarrow \pm i$. So
$$Mm(P_n)_{\C}\approx\left[\begin{matrix}
2x_n&2&\pm 2ix_n&0 \\1&\pm ix_n&\pm i& \pm i\\0&0&0&1\\
2\overline{x_n}^2&-2\overline{(A\pm iq)}&\pm 2\overline{ix_n^2}&0\end{matrix}\right].$$
Observe that $\det Mm(P_n)_{\C}\approx 8|x_n|^4$, moreover, the norm of every $3\times 3$ minor of $Mm(P_n)_{\C}$ is in $O(|x_n|^4)$. It follows that the norm of every entry of $Mm(P_n)^{-1}_{\C}$, and consequently of every entry of $Mm(P_n)^{-1}_{\R}$, is in $O(1)$.

Next, consider the case $(x_n:y_n:z_n)\rightarrow(0:1:0)$. Since $y_n(z_n+A_n)+x_n\rightarrow q$ we must have $z_n\rightarrow -A$ and since $x_n^2+z_n^2+2y_n\rightarrow p$ we must have $x_n\approx \pm\sqrt{-2y_n}$. So
$$Mm(P_n)_{\C}\approx\left[\begin{matrix}
\pm 2\sqrt{-2y_n}&2&-2A&0\\1&o(1)&y_n& y_n\\0&0&0&1\\
2\overline{y_n}&\mp 2\overline{\sqrt{-2y_n}y_n}&o(\sqrt{|y_n|})&0\end{matrix}\right].$$
Observe that $\det Mm(P_n)_{\C}\approx -8|y_n^3|$, moreover, the norm of every $3\times 3$ minor of $Mm(P_n)_{\C}$ is in $O(|y_n|^{3})$.  It follows that the norm of every entry of $Mm(P_n)^{-1}_{\C}$, and consequently of every entry of $Mm(P_n)^{-1}_{\R}$, is in $O(1)$. Since $||P_n||\approx|y_n|$ we obtain the desired inequality.
\end{proof}

Observe that for $A\neq 0$ we have
$$\left(\frac{1}{A^2}p,\frac{1}{A}q\right)\circ F_A\circ (Ax,y,Az)=(x^2+z^2+2A^{-2}y,yz+x+y).$$
Moreover,
$$(x^2+z^2+2A^{-2}y,yz+x+y)=(p,-q)\circ(x^2+z^2-2A^{-2}y,yz+x+y)\circ(-x,-y,z).$$
Thus, if $A^4=A_1^4$ then $F_A$ is linearly equivalent to $F_{A_1}$. Consequently, if $A^4=-16$ then $F_A$ is linearly equivalent to
\begin{equation}\label{F_gen6}
F_9=(x^2+z^2+2y,yz+x+\sqrt{2}(1+i)y).
\end{equation}

We return to equation \eqref{F_gen4} and consider the case $a_8b_7=0$.

For $b_7=0$ and $a_8b_8\neq 0$ we obtain
$$F_{10}=(x^2+z^2+y,yz+y)=(b_8^{-2}p,a_8 b_8^{-3} q)\circ(x^2+z^2+a_8y,yz+b_8y)\circ(b_8x,a_8^{-1}b_8^2 y, b_8 z).$$
For $b_7=0$ and $a_8b_8=0$ we obtain
$$F_{12}=(x^2+z^2,yz+y),\quad F_{13}=(x^2+z^2+y,yz),\quad F_{15}=(x^2+z^2,yz).$$
We are left with $b_7\neq 0$ and $a_8=0$. Depending on $b_8$ being zero or not we obtain:
$$F_{11}=(x^2+z^2,yz+x+y),\quad F_{14}=(x^2+z^2,yz+x).$$

\section{Other behavior at infinity}\label{sec_specinf}

In this section we  consider mappings $F$ with nongeneric behavior at infinity. We enumerate the linear and topological equivalence classes of mappings $F$ such that $\tl{F}\neq \tl{F}_1$ and $\tl{F}\neq \tl{F}_2$.

First, take $F$ such that $\tl{F}=\tl{F}_3$. We have
\begin{equation}
F=(x^2+y^2+a_7x+a_8y+a_9z+a_{10},z^2+b_7x+b_8y+b_9z+b_{10}).
\end{equation}

By composing with translations we can assume that $a_7=a_8=b_9=a_{10}=b_{10}=0$ and obtain
\begin{equation}
F=(x^2+y^2+a_9z,z^2+b_7x+b_8y).
\end{equation}

Now we have three cases regarding the linear form $b_7x+b_8y$: it is nonzero and does not divide $x^2+y^2$, it does divide $x^2+y^2$ or it is zero. After linearly changing the $x$ and $y$ coordinates we obtain one of the following forms:
\begin{equation}
(x^2+y^2+a_9z,z^2+x),\quad (xy+a_9z,z^2+x),\quad (x^2+y^2+a_9z,z^2).
\end{equation}
If $a_9\neq 0$ then by composing with $(a_9^{2/3} x,a_9^{2/3} y, a_9^{1/3} z)$ and $(a_9^{-4/3} p,a_9^{-2/3} q)$ we obtain $a_9=1$. Thus we obtain the following corollary:

\begin{corollary}
Let $F$ such that $\tl{F}=\tl{F}_3$. Then $F$ is linearly equivalent to one of the following:
\begin{enumerate}
\item $F_{16}=(x^2+y^2+z,z^2+x)$,
\item $F_{17}=(x^2+y^2,z^2+x)$,
\item $F_{18}=(xy+z,z^2+x)$,
\item $F_{19}=(xy,z^2+x)$,
\item $F_{20}=(x^2+y^2+z,z^2)$,
\item $F_{21}=(x^2+y^2,z^2)$.
\end{enumerate}
\end{corollary}

Next, take $F$ such that $\tl{F}=\tl{F}_4$. We have
\begin{equation}
F=(x^2+2yz+a_7x+a_8y+a_9z+a_{10},y^2+2xy+b_7x+b_8y+b_9z+b_{10}).
\end{equation}

By composing with translations we can assume that $a_7=a_8=a_9=a_{10}=b_{10}=0$ and obtain
\begin{equation}\label{F_gen7}
F=(x^2+2yz,y^2+2xy+2b_7x+2b_8y+2b_9z).
\end{equation}
Now assume that $b_9\neq 0$. Composing with $(b_9 x, b_9 y, b_9 z)$ and $(b_9^{-2} p,b_9^{-2} q)$ we obtain
\begin{equation}
F_{A,B}=(x^2+2yz,y^2+2xy+2Ax+2By+2z).
\end{equation}

Note that $C(F_{A,B})=V(x-y^2-Ay,z-xy-y^2-By)$, it is a smooth curve of degree $3$ that can be easily parametrized by
\begin{equation}\label{phi2}
\varphi:\C\ni t\mapsto\left(t^2+At,t,t^3+(A+1)t^2+Bt\right)\in C(F_{A,B}).
\end{equation}

The critical set of of $F_{A,B}\circ\varphi$ is equal to the set of zeroes of the derivative of $g\circ\varphi$, which is equal to the set of zeroes of the polynomial
\begin{equation}\label{Hc2}
H_{c2}=6t^2+(6A+3)t+(A^2+2B).
\end{equation}
Observe that $H_{c2}$ has two distinct roots if and only if $4A^2+12A-16B+3\neq 0$. So let us assume that $H_{c2}$ has two distinct roots. We denote the discriminant and one of the roots of $H_{c2}$ by $T=(4/3\cdot A^2+4A-16/3\cdot B+1)^{1/2}$ and $S=1/4\cdot(-2A-1+T)$. Take
$$R_1=T^2x+\frac{T^2-T}{2}y+S^2+AS,$$
$$R_2=Ty+S,$$
$$R_3=-(S+A)T^2x-\frac{1}{8}T(2AT-T+1)y+T^3z+\frac{1}{8}S(2A-T^2+T).$$
We have
\begin{equation}\label{aut}
F_{22}=F_{0,0}=\left(\frac{p-Sq}{T^4}-a_{10},\frac{q}{T^3}-b_{10}\right)\circ F_{A,B}\circ(R_1,R_2,R_3).
\end{equation}
The constants $a_{10}$ and $b_{10}$ in equation \eqref{aut} can be computed in an elementary way. We do not include them because they are complicated and provide no insight.

Now assume that $H_{c2}$ has a double root, i.e., $16B=4A^2+12A+3$. Take
$$R_4=-\frac{A}{4}(2x+y)+z-\frac{1}{32}(4A^2+3A).$$
We have
\begin{equation}\label{aut2}
F_{23}=F_{0,3/16}=\left(p+\frac{A}{2}q-a_{10},q-b_{10}\right)\circ F_{A,B}\circ\left(x-\frac{A^2}{4},y-\frac{A}{2},R_4\right)
\end{equation}
where $a_{10}=-1/32\cdot(6A^4+8A^3+3A^2)$ and $b_{10}=-1/8\cdot(4A^3+6A^2+3A)$.

Now we return to equation \eqref{F_gen7} and assume that $b_9=0$. If $b_7\neq 0$ then we compose $F$ with
$(b_t^{-2}p+b_7^{-3}b_8q+b_7^{-2}b_8^2, b_7^{-2}q+2b_7^{-1}b_8)$ and $(b_7x-b_8,b_7y,-b_8/2\cdot(2x+y)+b_7z)$ and obtain
\begin{equation}
F_{24}=(x^2+2yz,y^2+2xy+2x).
\end{equation}
If $b_7=0$ and $b_8\neq 0$ then we compose $F$ with
$(b_8^{-2}p, b_8^{-2}q)$ and $(b_8x,b_8y,b_8z)$ and obtain
\begin{equation}
F_{25}=(x^2+2yz,y^2+2xy+2y).
\end{equation}
Finally, if both $b_7=0$ and $b_8= 0$ then we have
\begin{equation}
F_{26}=(x^2+2yz,y^2+2xy).
\end{equation}

We gather all the cases that occur in equation \eqref{F_gen7} in the following corollary:

\begin{corollary}
Let $F$ such that $\tl{F}=\tl{F}_4$. Then $F$ is linearly equivalent to one of the following:
\begin{enumerate}
\item $F_{22}=(x^2+2yz,y^2+2xy+2z)$,
\item $F_{23}=(x^2+2yz,y^2+2xy+3/8\cdot y+2z)$,
\item $F_{24}=(x^2+2yz,y^2+2xy+2x)$,
\item $F_{25}=(x^2+2yz,y^2+2xy+2y)$,
\item $F_{26}=(x^2+2yz,y^2+2xy)$.
\end{enumerate}
\end{corollary}

Next, take $F$ such that $\tl{F}=\tl{F}_5$. We have
\begin{equation}
F=(x^2+2yz+a_7x+a_8y+a_9z+a_{10},z^2+b_7x+b_8y+b_9z+b_{10}).
\end{equation}

By composing with translations we can assume that $a_7=a_8=a_9=a_{10}=b_{10}=0$ and obtain
\begin{equation}\label{F_gen8}
F=(x^2+2yz,z^2+2b_7x+2b_8y+2b_9z).
\end{equation}

Now assume that $b_8\neq 0$. Compose $F$ with $(b_8^{-2}p,b_8^{-2}q)$ and $(b_8x,b_8y,b_8z)$ to obtain $F=(x^2+2yz,z^2+2b_7x+2y+2b_9z)$. Denote $T=1/3\cdot(b_7^2+2b_9)$ and compose $F$ with $(p+Tq+2T^3,q+3T^2)$ and $(x+b_7z-b_7T,-b_7x+y-(2T-b_9)z+T(T-b_9),z-T)$ to obtain
\begin{equation}
F_{27}=(x^2+2yz,z^2+2y).
\end{equation}

Now return to equation \eqref{F_gen8} and assume that $b_8=0$. If $b_7\neq 0$ then compose $F$ with $(b_7^{-2}p,b_7^{-2}q)$ and $(b_7x-b_9z,b_9x+b_7y-1/2\cdot b_7^{-1}b_9^2 z, b_7 z)$ to obtain
\begin{equation}
F_{28}=(x^2+2yz,z^2+2x).
\end{equation}

Now assume that $b_8=b_7=0$ in equation \eqref{F_gen8}. If $b_9\neq 0$ then composing $F$ with $(b_9^{-2}p,b_9^{-2}q)$ and $(b_9x,b_9y,b_9z)$ we obtain 
\begin{equation}
F_{29}=(x^2+2yz,z^2+2z).
\end{equation}

Otherwise we have
\begin{equation}
F_{30}=(x^2+2yz,z^2).
\end{equation}

We gather all the cases that occur in equation \eqref{F_gen8} in the following corollary:

\begin{corollary}
Let $F$ such that $\tl{F}=\tl{F}_5$. Then $F$ is linearly equivalent to one of the following:
\begin{enumerate}
\item $F_{27}=(x^2+2yz,z^2+2y)$,
\item $F_{28}=(x^2+2yz,z^2+2x)$,
\item $F_{29}=(x^2+2yz,z^2+2z)$,
\item $F_{30}=(x^2+2yz,z^2)$.
\end{enumerate}
\end{corollary}

Next, take $F$ such that $\tl{F}=\tl{F}_6$. We have
\begin{equation}\label{F_gen9}
F=(x^2+a_7x+a_8y+a_9z+a_{10},y^2+b_7x+b_8y+b_9z+b_{10}).
\end{equation}
First assume that $a_9b_9\neq 0$. Composing $F$ with $(x,y,-b_7b_9^{-1}x-a_8a_9^{-1}y+z)$ we obtain $F=(x^2+a_7x+a_9z+a_{10},y^2+b_8y+b_9z+b_{10})$. Next, we compose $F$ with $(x-a_7/2,y-b_8/2,z)$ to obtain $F=(x^2+a_9z+a_{10},y^2+b_9z+b_{10})$. And finally we compose with $(a_9^{-1}(p-a_{10}),b_9^{-1}(q-b_{10}))$ and $(a_9^{1/2}x,b_9^{1/2}y,2z)$ to obtain
\begin{equation}
F_{31}=(x^2+2z,y^2+2z).
\end{equation}

Now, assume that in equation \eqref{F_gen9} exactly one of $a_9$ and $b_9$ is nonzero. By symmetry we may assume $a_9\neq 0$ and $b_9=0$. Composing with $(x,y-b_8/2,a_9^{-1}(-a_7x-a_8y+2z))$ we obtain $F=(x^2+2z+a_{10},y^2+b_7x+b_{10})$. Obviously, depending on whether $b_7\neq 0$ or $b_7=0$ we obtain
\begin{equation}
F_{32}=(x^2+2z,y^2+2x)
\end{equation}
or
\begin{equation}
F_{33}=(x^2+2z,y^2).
\end{equation}

The last case to consider in equation \eqref{F_gen9} is $a_9=b_9=0$. In this case $F$ does not depend on the $z$ variable and we may use the list in \cite{fj} to obtain that $F$ is affinely equivalent to one of the following:
\begin{equation}
F_{34}=(x^2+2y,y^2+2x),\quad F_{35}=(x^2+2y,y^2),\quad F_{36}=(x^2,y^2). 
\end{equation}

We gather all the cases that occur in equation \eqref{F_gen9} in the following corollary:

\begin{corollary}
Let $F$ such that $\tl{F}=\tl{F}_6$. Then $F$ is linearly equivalent to one of the following:
\begin{enumerate}
\item $F_{31}=(x^2+2z,y^2+2z)$,
\item $F_{32}=(x^2+2z,y^2+2x)$,
\item $F_{33}=(x^2+2z,y^2)$,
\item $F_{34}=(x^2+2y,y^2+2x)$,
\item $F_{35}=(x^2+2y,y^2)$,
\item $F_{36}=(x^2,y^2)$.
\end{enumerate}
\end{corollary}

Next, take $F$ such that $\tl{F}=\tl{F}_7$. We have
\begin{equation}
F=(xy+a_7x+a_8y+a_9z+a_{10},yz+b_7x+b_8y+b_9z+b_{10}).
\end{equation}

By composing with translations we can assume that $a_7=a_8=b_8=a_{10}=b_{10}=0$ and obtain
\begin{equation}\label{F_gen10}
F=(xy+a_9z,yz+b_7x+b_9z).
\end{equation}

First assume that $a_9b_7\neq 0$. By composing with $(a_9^{-1}b_7^{-1/2}p,a_9^{-1/2}b_7^{-1}q)$ and $(\sqrt{a_9}x,\sqrt{a_9b_7}y,\sqrt{b_7}z)$ we obtain $F=(xy+z,yz+x+b_9z)$. Observe that $C(F)=V(y^2+b_9y-1,x-yz)$. If $b_9^2+4\neq 0$ then $C(F)$ consists of two lines intersecting at infinity. If $b_9^2+4=0$ then $C(F)$ is a double line. Let us assume that $b_9^2+4\neq 0$, i.e., $y^2+b_9y-1=0$ has two distinct roots. We denote the discriminant and roots by $T=\sqrt{b_9^2+4}$, $S_1=(-b_9+T)/2$ and $S_2=(-b_9+T)/2$. Let
\begin{equation}
F_{37}=(xy,yz+z).
\end{equation}

We have
\begin{equation}
F_{37}=(S_2p+q,S1p+q)\circ F\circ((-x+z)T^{-2},Ty+S_1,(S_1x-S_2z)T^{-2}).
\end{equation}

Now consider $b_9^2+4=0$. If $b_9=-2i$ then we compose $F$ with $(p,-q)$ and $(-x,-y,z)$ and obtain $b_9=2i$, i.e., $F=(xy+z,yz+x+2iz)$. We compose $F$ with $(p,-ip+q)$ and $(x,y-i,ix+z)$ and obtain
\begin{equation}
F_{38}=(xy+z,yz).
\end{equation}

Now return to equation \eqref{F_gen10} and assume $a_9=0$ and $b_7\neq 0$. If $b_9=0$, i.e., $F=(xy,yz+b_7x)$, then we compose $F$ with $(-b_7p+q,b_7p)$ and $(b_7^{-1}z,y,x+z)$ and obtain $F_{38}$. If $b_9\neq 0$, i.e., $F=(xy,yz+b_7x+b_9z)$ with $b_7b_9\neq 0$, then we compose $F$ with $(p,b_7p+b_9q)$ and $(b_9^{-1}x,b_9y,b_9^{-2}(-b_7x+z))$ and obtain $F_{37}$.

Now return to equation \eqref{F_gen10} and assume $a_9\neq 0$ and $b_7= 0$. If $b_9=0$, i.e., $F=(xy+a_9z,yz)$, then we compose $F$ with $(a_9^{-1}p,q)$ and $(a_9x,y,z)$ and obtain $F_{38}$. If $b_9\neq 0$, i.e., $F=(xy+a_9z,yz+b_9z)$ with $a_9b_9\neq 0$, then we compose $F$ with $(q,-b_9p+a_9q)$ and $(b_9^{-2}(-a_9x+z),-b_9y,-b_9^{-1}x)$ and obtain $F_{37}$.

Finally, we are left with the case $a_9=b_7=0$ in equation \eqref{F_gen10}. If $b_9\neq 0$ then we compose $F$ with $(b_9^{-1}p,b_9^{-1}q)$ and $(x,b_9y,z)$ and obtain $F_{37}$. If $b_9=0$ then we have
\begin{equation}
F_{39}=(xy,yz).
\end{equation}

We gather all the cases that occur in equation \eqref{F_gen10} in the following corollary:

\begin{corollary}
Let $F$ such that $\tl{F}=\tl{F}_7$. Then $F$ is linearly equivalent to one of the following:
\begin{enumerate}
\item $F_{37}=(xy,yz+z)$,
\item $F_{38}=(xy+z,yz)$,
\item $F_{39}=(xy,yz)$.
\end{enumerate}
\end{corollary}

Next, take $F$ such that $\tl{F}=\tl{F}_8$. We have
\begin{equation}\label{F_gen11}
F=(xy+a_7x+a_8y+a_9z+a_{10},y^2+b_7x+b_8y+b_9z+b_{10}).
\end{equation}
First assume that $b_9\neq 0$. Composing $F$ with $(x-a_8,y,b_9^{-1}(-b_7x-b_8y+z))$ we obtain $F=(xy+a_7x+a_9z+a_{10},y^2+z+b_{10})$. Next, we compose $F$ with $(p-a_9(q-a_7^2-b_{10})-a_{10},q-a_7^2-b_{10})$ and $(x+a_9y-2a_7a_9,y-a_7,2a_7y+2z)$ to obtain:
\begin{equation}
F_{40}=(xy,y^2+2z).
\end{equation}

Now return to equation \eqref{F_gen11} and assume $b_9=0$ and $a_9\neq 0$. We compose $F$ with $(x,y-b_8/2,a_9^{-1}((-a_7+b_8/2)x-a_8y+z))$ and obtain $F=(xy+z+a_{10},y^2+b_7x+b_{10})$. If $b_7\neq 0$ then by composing with $(b_7(p-a_{10}),q-b_{10})$ and $(b_7^{-1}x,y,b_7^{-1}z)$ we obtain
\begin{equation}
F_{41}=(xy+z,y^2+x).
\end{equation}
If $b_7= 0$ then we obtain
\begin{equation}
F_{42}=(xy+z,y^2).
\end{equation}

Finally, if $b_9=a_9=0$ in equation \eqref{F_gen11} then $F$ does not depend on the $z$ variable and we may use the list in \cite{fj} to obtain that $F$ is affinely equivalent to one of the following:
\begin{equation}
F_{43}=(xy,y^2+2x),\quad F_{44}=(xy,y^2+2y),\quad F_{45}=(xy,y^2). 
\end{equation}

We gather all the cases that occur in equation \eqref{F_gen11} in the following corollary:

\begin{corollary}
Let $F$ such that $\tl{F}=\tl{F}_8$. Then $F$ is linearly equivalent to one of the following:
\begin{enumerate}
\item $F_{40}=(xy,y^2+2z)$,
\item $F_{41}=(xy+z,y^2+x)$,
\item $F_{42}=(xy+z,y^2)$,
\item $F_{43}=(xy,y^2+2x)$,
\item $F_{44}=(xy,y^2+2y)$,
\item $F_{45}=(xy,y^2)$.
\end{enumerate}
\end{corollary}

This concludes the classification of all $F$ which are not linearly equivalent to a mapping with a component of degree lower than $2$. We continue the classification with $F$ such that $\tl{F}=\tl{F}_9$. We have
\begin{equation}
F=(x^2+yz+a_7x+a_8y+a_9z+a_{10},x+b_{10}).
\end{equation}

By composing $F$ with translations we can assume that $a_7=a_8=a_9=a_{10}=b_{10}=0$ and obtain
\begin{equation}
F_{46}=(x^2+yz,x).
\end{equation}
Similarly, for $\tl{F}=\tl{F}_{10}$ and $\tl{F}=\tl{F}_{11}$ we obtain
\begin{equation}
F_{47}=(x^2+yz,y), \quad F_{48}=(x^2+y^2+z^2,0).
\end{equation}

Next, take $F$ such that $\tl{F}=\tl{F}_{12}$. We have
\begin{equation}
F=(x^2+y^2+a_7x+a_8y+a_9z+a_{10},z+b_{10}).
\end{equation}
We compose $F$ with $(x-a_7/2,y-a_8/2,z-b_{10})$ an obtain $F=(x^2+y^2+a_9z+a_{10},z)$. Next we compose $F$ with $(p-a_9q-a_{10},q)$ and obtain 
\begin{equation}
F_{49}=(x^2+y^2,z).
\end{equation}

For $F$ such that $\tl{F}=\tl{F}_{13}$ we have
\begin{equation}
F=(x^2+y^2+a_7x+a_8y+a_9z+a_{10},x+b_{10}).
\end{equation}
By composing $F$ with translations we can assume that $a_7=a_8=a_{10}=b_{10}=0$ and obtain $F=(x^2+y^2+a_9z,x)$. Depending on whether $a_9=0$ we obtain one of the following two:
\begin{equation}
F_{50}=(x^2+y^2+z,x), \quad F_{51}=(x^2+y^2,x).
\end{equation}
Similarly, for $\tl{F}=\tl{F}_{14}$ we obtain
\begin{equation}
F_{52}=(xy+z,x), \quad F_{53}=(xy,x)
\end{equation}
and for $\tl{F}=\tl{F}_{15}$ we obtain
\begin{equation}
F_{54}=(x^2+y^2+z,0), \quad F_{55}=(x^2+y^2,0).
\end{equation}

For $F$ such that $\tl{F}=\tl{F}_{16}$ we have
\begin{equation}
F=(x^2+a_7x+a_8y+a_9z+a_{10},y+b_{10}).
\end{equation}
By composing $F$ with translations we can assume that $a_7=a_{10}=b_{10}=0$ and obtain $F=(x^2+a_8y+a_9z,y)$. Composing with $(p-a_8q,q)$ we obtain $a_8=0$ and depending on whether $a_9=0$ we obtain one of the following two:
\begin{equation}
F_{56}=(x^2+z,y), \quad F_{57}=(x^2,y).
\end{equation}

For $F$ such that $\tl{F}=\tl{F}_{17}$ we have
\begin{equation}
F=(x^2+a_7x+a_8y+a_9z+a_{10},x+b_{10}).
\end{equation}
By composing $F$ with translations we can assume that $a_7=a_{10}=b_{10}=0$ and obtain $F=(x^2+a_8y+a_9z,x)$. Assume that $a_8\neq 0$ or $a_9\neq 0$, by symmetry we can take $a_8\neq 0$. Composing with $(x,a_8^{-1}(y-a_9z),z)$ we obtain 
\begin{equation}
F_{58}=(x^2+y,x).
\end{equation}
On the other hand, for $a_8=a_9=0$ we have
\begin{equation}
F_{59}=(x^2,x).
\end{equation}
Similarly, for $\tl{F}=\tl{F}_{18}$ we obtain
\begin{equation}
F_{60}=(x^2+y,0), \quad F_{61}=(x^2,0).
\end{equation}

Finally for $F$ equal to $\tl{F}_{19}$, $\tl{F}_{20}$ or $\tl{F}_{21}$ we must have $F=(\tl{f}+a_{10},\tl{g}+b_{10})$, so after composing with a translation we obtain
\begin{equation}
F_{62}=(x,y), \quad F_{63}=(x,0), \quad F_{64}=(0,0).
\end{equation}

\vspace{5mm}

{\bf Acknowledgments.}
I would like to thank professor Zbigniew Jelonek and professor Maria Aparecida Soares Ruas for many helpful discussions.

\end{document}